\documentclass[a4paper]{amsart}
\usepackage[utf8]{inputenc}
\usepackage{geometry}
\usepackage{amsmath}
\usepackage{amsfonts}
\usepackage{amssymb}
\usepackage{amsthm}
\usepackage{braket}
\usepackage{mathtools}
\usepackage{bm}
\usepackage[all]{xy}
\usepackage[unicode=true]{hyperref}

\theoremstyle{definition}
\newtheorem{defn}{Definition}
\theoremstyle{plain}
\newtheorem{teo}{Theorem}
\newtheorem{lem}{Lemma}
\theoremstyle{remark}
\newtheorem{rem}{Remark}
\newtheorem{ex}{Example}

\newcommand{\Q}{\mathbf{Q}}
\newcommand{\deq}{\mathrel{:=}}
\newcommand{\isom}{\simeq}
\newcommand{\p}[1]{\left(#1\right)}
\newcommand{\de}{\mathrm{d}}
\newcommand{\bb}[1]{\mathbb{#1}}
\newcommand{\app}[3]{#1\colon #2\to #3}
\newcommand{\ol}[1]{\overline{#1}}
\newcommand{\alg}{\mathrm{alg}}
\newcommand{\op}{\mathrm{op}}

\newcommand{\DR}{\mathrm{DR}}
\mathchardef\mhyphen="2D
\newcommand{\bimod}[1]{{#1\mhyphen\mathbf{Bimod}}}
\newcommand{\comm}[2]{[#1,#2]}
\newcommand{\rist}[2]{\left. #1\right|_{#2}}

\newcommand{\floor}[1]{\left\lfloor #1\right\rfloor}
\newcommand{\ceil}[1]{\left\lceil #1\right\rceil}
\newcommand{\trasp}[1]{\prescript{t}{}{#1}}
\newcommand{\falg}[1]{\bb{C}\langle #1\rangle}
\newcommand{\res}{\mathrm{res}}
\newcommand{\red}{\mathrm{red}}

\DeclareMathOperator{\id}{id}
\DeclareMathOperator{\diag}{diag}
\DeclareMathOperator{\supp}{supp}
\DeclareMathOperator{\tr}{Tr}

\DeclareMathOperator{\GL}{GL}

\DeclareMathOperator{\ASL}{ASL}

\DeclareMathOperator{\PGL}{PGL}

\DeclareMathOperator{\Mat}{Mat}
\DeclareMathOperator{\Aut}{Aut}
\DeclareMathOperator{\TAut}{TAut}

\DeclareMathOperator{\Aff}{Aff}
\DeclareMathOperator{\ST}{Tri}
\newcommand{\STl}{\ST^{\ell}}
\DeclareMathOperator{\oST}{opTri}
\newcommand{\oSTl}{\oST^{\ell}}
\DeclareMathOperator{\Rep}{Rep}
\DeclareMathOperator{\Der}{Der}

\begin{document}

\title{On a family of quivers related to the Gibbons-Hermsen system}
\author{Alberto Tacchella}
\address{ICMC - Universidade de S\~ao Paulo\\
  Avenida Trabalhador S\~ao-carlense, 400\\
  13566-590 S\~ao Carlos - SP, Brasil}
\email{altacch@gmail.com}

\begin{abstract}
  We introduce a family of quivers \(Z_{r}\) (labeled by a natural number
  \(r\geq 1\)) and study the non-commutative symplectic geometry of the
  corresponding doubles \(\Q_{r}\). We show that the group of non-commutative
  symplectomorphisms of the path algebra \(\bb{C}\Q_{r}\) contains two copies
  of the group \(\GL_{r}\) over a ring of polynomials in one indeterminate,
  and that a particular subgroup \(\mathcal{P}_{r}\) (which contains both of
  these copies) acts on the completion \(\mathcal{C}_{n,r}\) of the phase
  space of the \(n\)-particles, rank \(r\) Gibbons-Hermsen integrable system
  and connects each pair of points belonging to a certain dense open subset of
  \(\mathcal{C}_{n,r}\). This generalizes some known results for the cases
  \(r=1\) and \(r=2\).
\end{abstract}

\maketitle

\section{Introduction}

\subsection{Gibbons-Hermsen manifolds}

For every \(n,r\in \bb{N}\) let us denote by \(\Mat_{n,r}(\bb{C})\) the
complex vector space of \(n\times r\) matrices with entries in \(\bb{C}\). We
consider the space
\begin{equation}
  \label{eq:def-Vnr}
  V_{n,r}\deq \Mat_{n,n}(\bb{C})\oplus \Mat_{n,n}(\bb{C})\oplus
  \Mat_{n,r}(\bb{C})\oplus \Mat_{r,n}(\bb{C}).
\end{equation}
Using the identification between \(\Mat_{n,r}(\bb{C})\) and the dual of
\(\Mat_{r,n}(\bb{C})\) provided by the bilinear form
\begin{equation}
  \label{eq:trace}
  (A,B)\mapsto \tr AB
\end{equation}
we can view \(V_{n,r}\) as the cotangent bundle of the vector space
\(\Mat_{n,n}(\bb{C})\oplus \Mat_{r,n}(\bb{C})\). In other words, denoting by
\((X,Y,v,w)\) a point in \(V_{n,r}\) we think of the pair \((Y,v)\) as a
cotangent vector applied at the point \((X,w)\). It follows that on
\(V_{n,r}\) a canonical (holomorphic) symplectic form is defined:
\begin{equation}
  \label{eq:omega}
  \omega(X,Y,v,w) = \sum_{i,j=1}^{n} \de y_{ji} \wedge \de x_{ij} +
  \sum_{i=1}^{n} \sum_{\alpha=1}^{r} \de v_{i\alpha}\wedge \de w_{\alpha i}.
\end{equation}
The group \(\GL_{n}(\bb{C})\) acts on \(V_{n,r}\) by
\[ g.(X,Y,v,w) = (gXg^{-1}, gYg^{-1}, gv, wg^{-1}). \]
This action is Hamiltonian, and the corresponding moment map
\(\app{\mu}{V_{n,r}}{\mathfrak{gl}_{n}(\bb{C})}\) is
\begin{equation}
  \label{eq:mom-mu}
  \mu(X,Y,v,w) = [X,Y] - vw.
\end{equation}
For every \(\tau\in \bb{C}^{*}\) the action of \(\GL_{n}(\bb{C})\) on
\(\mu^{-1}(\tau I)\) is free; the corresponding Marsden-Weinstein quotient,
\begin{equation}
  \label{eq:Cnr}
  \mathcal{C}_{n,r}\deq \mu^{-1}(\tau I)/\GL_{n}(\bb{C}),
\end{equation}
is a smooth symplectic manifold of dimension \(2nr\) that we call the
(\(n\)-particle, rank \(r\)) \emph{Gibbons-Hermsen manifold}. Different
choices of \(\tau\) yield isomorphic symplectic manifolds; from now on we
simply suppose that some choice has been made and stick to it for the rest of
the paper.

The manifold \(\mathcal{C}_{n,r}\) can be seen naturally as a completion of
the phase space of the \(n\)-particle, rank \(r\) integrable system introduced
by Gibbons and Hermsen in \cite{gh84} as a generalization of the well-known
rational Calogero-Moser model. In more detail, let us denote by
\(\mathcal{C}_{n,r}'\) the subset of \(\mathcal{C}_{n,r}\) consisting of the
orbits of those quadruples in which the matrix \(X\) has \(n\) distinct
eigenvalues, to be interpreted as the positions of the \(n\) particles on the
complex plane. In each such orbit we can find a point \((X,Y,v,w)\) such that
\(X\) is diagonal, say \(X = \diag (x_{1}, \dots, x_{n})\), and
\[ Y =
\begin{pmatrix}
  y_{1} & \frac{v_{1\bullet}w_{\bullet 2}}{x_{1}-x_{2}} & \hdots &
  \frac{v_{1\bullet}w_{\bullet n}}{x_{1}-x_{n}}\\
  \frac{v_{2\bullet}w_{\bullet 1}}{x_{2}-x_{1}} & y_{2} & \ddots &
  \vdots\\
  \vdots & \ddots & \ddots & \vdots\\
  \frac{v_{n\bullet}w_{\bullet 1}}{x_{n}-x_{1}} & \hdots & \hdots & y_{n}
\end{pmatrix} \]
where \(v_{i\bullet}\) denotes the \(r\)-components row vector given by the
\(i\)-th row of the matrix \(v\) and similarly \(w_{\bullet j}\) denotes the
\(r\)-components column vector given by the \(j\)-th column of the matrix
\(w\); the diagonal entries \((y_{1}, \dots, y_{n})\in \bb{C}^{n}\) are free.
If we fix the ordering of the eigenvalues of \(X\), this representative is
unique up to the action of a diagonal matrix \(\diag (\lambda_{1}, \dots,
\lambda_{n})\in \GL_{n}(\bb{C})\); such a matrix fixes the parameters
\((x_{1}, \dots, x_{n})\) and \((y_{1}, \dots, y_{n})\) and acts as follows on
the vectors \(v_{i\bullet}\) and \(w_{\bullet i}\):
\begin{equation}
  \label{eq:act-vw}
  v_{i\bullet}\mapsto \lambda_{i}v_{i\bullet} \qquad w_{\bullet i}\mapsto
  w_{\bullet i} \lambda_{i}^{-1}.
\end{equation}
The restriction of the symplectic form \eqref{eq:omega} on
\(\mathcal{C}_{n,r}'\) reads
\[ \rist{\omega}{\mathcal{C}_{n,r}'} = \sum_{i=1}^{n} \left( \de y_{i}\wedge
  \de x_{i} + \sum_{\alpha=1}^{r} \de v_{i\alpha} \wedge \de w_{\alpha i}\right) \]
and this shows that the coordinates
\[ (x_{1}, \dots, x_{n}, w_{11}, \dots, w_{rn}, y_{1}, \dots, y_{n}, v_{11},
\dots, v_{nr}) \]
are canonical. We conclude that \(\mathcal{C}_{n,r}'\) can be interpreted as
the phase space of a system of \(n\) point particles of equal mass located at
the points \(x_{i}\) with momenta \(y_{i}\), each particle having some
internal degrees of freedom parametrized by a \(r\)-component
covector-vector pair \((v_{i\bullet},w_{\bullet i})\) living in the complex
manifold of dimension \(2r-2\) defined by
\[ \mathcal{V}_{r}\deq \set{(\xi, \eta)\in \Mat_{1,r}(\bb{C})\times
  \Mat_{r,1}(\bb{C}) |  \xi \eta = -\tau}/\bb{C}^{*}, \]
where the action of \(\lambda\in \bb{C}^{*}\) is given by \(\lambda.(\xi,\eta)
= (\lambda\xi, \eta\lambda^{-1})\), as in \eqref{eq:act-vw}. The condition
\(v_{i\bullet} w_{\bullet i} = -\tau\) comes from the diagonal part of the
moment map equation \([X,Y] - vw = \tau I\); when \(r=1\) this completely
fixes the internal degrees of freedom, and one recovers the usual description
for the (complexified) rational Calogero-Moser system \cite{wils98}.

The \emph{Gibbons-Hermsen hierarchy} is defined by the family of
\(\GL_{n}(\bb{C})\)-invariant Hamiltonians
\begin{equation}
  \label{eq:20}
  J_{k,m}\deq \tr Y^{k} v m w
\end{equation}
with \(k\in \bb{N}\), \(m\in \Mat_{r,r}(\bb{C})\). The equations of motion
induced by \(J_{k,m}\) are
\begin{equation}
  \label{eq:eq-mot}
  \begin{array}{l}
    \displaystyle \dot{X} = \sum_{i=1}^{k} Y^{k-i} vmw Y^{i-1}\\
    \displaystyle \dot{Y} = 0\\
    \dot{v} = Y^{k}vm\\
    \dot{w} = -mwY^{k}
  \end{array} 
\end{equation}
In particular \(J_{2,I}\) is (up to scalar factors) the Hamiltonian considered
by Gibbons and Hermsen in \cite{gh84}.

It is not difficult to see that the flows determined by the evolution
equations \eqref{eq:eq-mot} are complete on \(\mathcal{C}_{n,r}\); this is the
reason for calling this manifold a completion of the phase space
\(\mathcal{C}_{n,r}'\). As in the Calogero-Moser case, these extended flows
correspond to motions which have been ``analytically continued'' through the
collisions.

\subsection{$\mathcal{C}_{n,r}$ as a quiver variety}

In order to investigate the geometry of the manifolds \(\mathcal{C}_{n,r}\) it
is useful to consider them as particular examples of \emph{quiver varieties},
a notion introduced by Nakajima in \cite{naka94}.

For the reader's convenience let us recall that a \emph{quiver} is simply a
directed graph, possibly with loops and multiple edges. We will only consider
quivers with a finite vertex set, say \(I = \{1, \dots, n\}\). A (complex)
representation of such a quiver \(Q\) is specified by giving for each vertex
\(i\in I\) a finite-dimensional complex vector space \(V_{i}\) and for each
edge \(\app{\xi}{i}{j}\) a linear map \(\app{V_{\xi}}{V_{i}}{V_{j}}\). If we
denote by \(d_{i}\) the dimension of the vector space \(V_{i}\), the vector
\(\bm{d} = (d_{1}, \dots, d_{n})\in \bb{N}^{n}\) is called the
\emph{dimension vector} of the representation \(V\). Clearly, representations
of \(Q\) with dimension vector \(\bm{d}\) correspond to points in the complex
vector space
\[ \Rep(Q,\bm{d})\deq \bigoplus_{i,j\in I} \bigoplus_{\app{\xi}{i}{j}}
\Mat_{d_{j},d_{i}}(\bb{C}). \]
On this space there is an action of the group
\begin{equation}
  \label{eq:gl-d}
  \GL_{\bm{d}}(\bb{C})\deq (\prod_{i\in I} \GL_{d_{i}}(\bb{C}))/\bb{C}^{*},
\end{equation}
where \(\bb{C}^{*}\) is seen as the subgroup of \(n\)-tuples of the form
\((\lambda I_{d_{1}}, \dots, \lambda I_{d_{n}})\) for each \(\lambda\in
\bb{C}^{*}\). This action is defined as follows: for every arrow
\(\app{\xi}{i}{j}\) the action of \([g_{1}, \dots, g_{n}]\in
\GL_{\bm{d}}(\bb{C})\) sends the matrix \(V_{\xi}\) to the matrix
\(g_{j}V_{\xi}g_{i}^{-1}\). In other words, each factor
\(\GL_{d_{i}}(\bb{C})\) acts on the \(d_{i}\)-dimensional space \(V_{i}\) by
change of basis.

The quotient
\begin{equation}
  \label{eq:naive-quot}
  \Rep(Q,\bm{d})/\GL_{\bm{d}}(\bb{C})
\end{equation}
parametrizes isomorphism classes of representations of the quiver \(Q\) with
dimension vector \(\bm{d}\). Unfortunately, as a topological space it is
usually quite pathological (e.g. not Hausdorff).

To improve the situation, one replaces the quiver \(Q\) with its \emph{double}
\(\ol{Q}\), obtained by keeping the same vertices and adding for each arrow
\(\app{\xi}{i}{j}\) a corresponding arrow \(\app{\xi^{*}}{j}{i}\) going in the
opposite direction. Using the isomorphisms provided by the bilinear form
\eqref{eq:trace} we can identify the representation space
\(\Rep(\ol{Q},\bm{d})\) with the cotangent bundle to \(\Rep(Q,\bm{d})\); then
\(\Rep(\ol{Q},\bm{d})\) is naturally a symplectic vector space. Moreover, the
action of the group \eqref{eq:gl-d} on this space coincides with the cotangent
lift of its action on the base. It follows that this action is Hamiltonian and
admits an equivariant moment map
\[ \app{\mu}{\Rep(\ol{Q},\bm{d})}{\mathfrak{gl}_{\bm{d}}(\bb{C})}, \]
where \(\mathfrak{gl}_{\bm{d}}(\bb{C})\) is the Lie algebra of
\(\GL_{\bm{d}}(\bb{C})\) (which we identify with its dual using once again the
bilinear form \eqref{eq:trace}). The idea is now to replace the badly-behaved
quotient \eqref{eq:naive-quot} with a symplectic quotient
\[ M_{\nu}\deq \mu^{-1}(\nu)/G_{\nu}, \]
where \(\nu\) denotes a point in \(\mathfrak{gl}_{\bm{d}}(\bb{C})\) and
\(G_{\nu}\) the stabilizer of \(\nu\) with respect to the adjoint action of
\(\GL_{\bm{d}}(\bb{C})\). If the action of \(G_{\nu}\) on the inverse image
\(\mu^{-1}(\nu)\subseteq \Rep(\ol{Q},\bm{d})\) is free and proper then the
quotient \(M_{\nu}\) is itself a smooth manifold, with a symplectic form
induced from the one on \(\Rep(\ol{Q},\bm{d})\).

At first sight, the definition \eqref{eq:Cnr} of \(\mathcal{C}_{n,r}\) seems to fit
precisely this construction by taking as \(Q\) the quiver
\[ \xymatrix{
  \bullet_{1} \ar@(ul,dl)[] \ar[r] & \bullet_{2}} \]
whose double \(\ol{Q}\) is
\begin{equation}
  \label{eq:q-wrong}
  \xymatrix{
    \bullet_{1} \ar@(u,l)[] \ar@(l,d)[] \ar@/_/[r] & \bullet_{2} \ar@/_/[l]}
\end{equation}
and considering representations with dimension vector \((n,r)\); then
\(\Rep(\ol{Q},(n,r)) = V_{n,r}\). However, the group \eqref{eq:gl-d} reads in
this case
\[ \GL_{(n,r)}(\bb{C}) = (\GL_{n}(\bb{C})\times \GL_{r}(\bb{C}))/\bb{C}^{*}, \]
hence it differs from \(\GL_{n}(\bb{C})\) as long as \(r>1\). To remedy
this problem, one should take instead a \emph{family} of quivers
\((Q_{r})_{r\geq 1}\) such that the corresponding doubles, call them
\(\Q_{r}\), have \(r\) distinct arrows \(1\to 2\) and \(r\) distinct arrows
\(2\to 1\):
\begin{equation}
  \label{eq:Qbar}
  \Q_{r} = \quad
  \xymatrix{
    \bullet_{1} \ar@(u,l)[] \ar@(l,d)[] 
    \ar@/_1.2pc/[r]|{}="a" \ar@/_0.5pc/[r]|{}="b" \ar@{.}"a";"b"_{r} &
    \bullet_{2} \ar@/_1.2pc/[l]|{}="c" \ar@/_0.5pc/[l]|{}="d" 
    \ar@{.}"c";"d"^{r}}
\end{equation}
and consider their representations with dimension vector \((n,1)\). Then the
space \(\Rep(\Q_{r},(n,1))\) can again be identified with \(V_{n,r}\) by
``slicing'' the matrices \(v\) and \(w\) in the \(r\) column matrices
\((v_{\bullet 1}, \dots, v_{\bullet r})\) and in the \(r\) row matrices
\((w_{1\bullet}, \dots, w_{r\bullet})\), respectively. The transition from the
quiver \eqref{eq:q-wrong} to the quiver \eqref{eq:Qbar} may be interpreted
also as the introduction of a \emph{framing} in the sense of Crawley-Boevey
\cite{cb01}.

\subsection{Non-commutative symplectic geometry}

One of the advantages in adopting this new point of view is that we can now
apply to the manifolds \(\mathcal{C}_{n,r}\) the powerful tools of
\emph{non-commutative symplectic geometry}. This theory was introduced by
Kontsevich in \cite{kont93} and developed, in particular with applications to
quiver varieties, by Ginzburg in \cite{ginz01} and Bocklandt-Le Bruyn in
\cite{blb02}.

The basic idea of (algebraic) non-commutative geometry is to generalize the
well-known duality between commutative rings and affine schemes to the much
larger class of associative, but not necessarily commutative, algebras. In
Section \ref{s:rem} below we will briefly review the fundamentals of this
approach, and in particular how one can define, starting from an associative
algebra \(A\), a complex \(\DR^{\bullet}(A)\) of ``non-commutative
differential forms'' on \(A\).

The connection with quiver varieties comes from the observation that to every quiver
\(Q\) one can associate in a natural manner an associative algebra
\(\bb{C}Q\), the \emph{path algebra} (over \(\bb{C}\)) \emph{of} \(Q\). This
is the complex associative algebra which is generated, as a linear space, by
all the (oriented) paths in \(Q\) and whose product is given by composition of
paths, or zero when two paths do not compose (meaning that the ending point of
the first does not coincide with the starting point of the second).

Suppose now that \(Q\) is actually the double of some other quiver. Then, as
it will be explained in Section \ref{s:rem}, on the corresponding path algebra
\(\bb{C}Q\) there is a canonical non-commutative 2-form \(\omega\in
\DR^{2}(\bb{C}Q)\) that plays the same r\^ole of a symplectic form in the
usual (i.e., commutative) symplectic geometry. Moreover, the group of
``non-commutative symplectomorphisms'' of \(\bb{C}Q\), that is algebra
automorphisms of \(\bb{C}Q\) that preserve (in a suitable sense) the
above-mentioned symplectic form, naturally acts on every quiver variety
derived from \(Q\).

Thus we can hope to obtain further information about the family of manifolds
\(\mathcal{C}_{n,r}\) by studying how the group of non-commutative
symplectomorphisms of \(\Q_{r}\) acts on them. For the case \(r=2\) this was
done in the second part of the paper \cite{bp08}, where it is proved in
particular that this action is \emph{transitive}. However, in order to endow
\(\Q_{r}\) with a non-commutative symplectic form one must choose for each
\(r\geq 1\) a quiver \(Q_{r}\) such that its double coincides with \(\Q_{r}\).
In other words, one has to decide which of the arrows in \eqref{eq:Qbar} are
the ``unstarred'' ones. As pointed out in \cite{bp08}, this choice is
\emph{not} irrelevant: different choices will give (in general) different
symplectomorphism groups. In Section \ref{s:nc} we define a family of
``zigzag'' quivers \((Z_{r})_{r\geq 1}\) by continuing the pattern started in
\cite{bp08} for \(r=2\):
\[ 
Z_{1} = \xymatrix{
  1\ar@(ul,dl)[]_{a} & 2 \ar[l]_{x_{1}}}
\qquad
Z_{2} = \xymatrix{
  1 \ar@(ul,dl)[]_{a} \ar@/_/[r]_{y_{1}} & 2 \ar@/_/[l]_{x_{1}}}
\qquad
Z_{3} = \xymatrix{
  1 \ar@(ul,dl)[]_{a} \ar[r]|{y_{1}} & 2 \ar@/_/[l]_{x_{1}} \ar@/^/[l]^{x_{2}}}
\]
\[ 
Z_{4} = \xymatrix{
  1 \ar@(ul,dl)[]_{a} \ar@/^0.25pc/[r]|{y_{1}} \ar@/_0.6pc/[r]_{y_{2}} &
  2 \ar@/_0.6pc/[l]_{x_{1}} \ar@/^0.25pc/[l]|{x_{2}}}
\qquad
Z_{5} = \xymatrix{
  1 \ar@(ul,dl)[]_{a} \ar@/^0.6pc/[r]|{y_{1}} \ar@/_0.6pc/[r]|{y_{2}} &
  2 \ar@/_1pc/[l]_{x_{1}} \ar[l]|{x_{2}} \ar@/^1pc/[l]^{x_{3}}}
\qquad
\dots
\]
In the main body of the paper we view the quiver \(\Q_{r}\) as the double of
\(Z_{r}\). Other possible choices are briefly discussed in Appendix
\ref{other-nc}.

\subsection{Aim and organization of the paper}

The aim of this paper is to study the non-commutative symplectic geometry of
the family of quivers \((\Q_{r})_{r\geq 1}\), with particular regard to its
group of symplectomorphisms, and to extend some of the results obtained in
\cite[Part 2]{bp08} and \cite{mt13} for the case \(r=2\) to higher values of
\(r\), hopefully clarifying their origin in the process.

We start in Section \ref{s:rem} by recalling some basic constructions in
non-commutative symplectic geometry that will be used in the sequel.

In Section \ref{s:gen} we study the properties of the path algebra of the
quiver \(\Q_{r}\) that do not depend on the particular non-commutative
symplectic structure chosen.

In Section \ref{s:nc} we consider \(\Q_{r}\) as the double of the zigzag
quiver \(Z_{r}\) and characterize the corresponding group of (tame)
symplectomorphisms \(\TAut(\bb{C}\Q_{r};c_{r})\).

In Section \ref{s:defP} we define a subgroup \(\mathcal{P}_{r}\subset
\TAut(\bb{C}\Q_{r};c_{r})\) which is the higher-rank version of the group
\(\mathcal{P} = \mathcal{P}_{2}\) considered in \cite{mt13}. As a
generalization (and strengthening) of one of the results proved there, we show
that this group contains two copies of the classical group \(\GL_{r}\) over
the complex polynomial ring in one indeterminate.

In Section \ref{s:action} we show that one can use the action of
\(\mathcal{P}_{r}\) on the manifolds \(\mathcal{C}_{n,r}\) to connect each
pair of points in the open subset
\begin{equation}
  \label{eq:2}
  \mathcal{R}_{n,r}\deq \mathcal{C}_{n,r}' \cup \mathcal{C}_{n,r}''
\end{equation}
where \(\mathcal{C}_{n,r}''\) is defined, by analogy with
\(\mathcal{C}_{n,r}'\), as the subset of \(\mathcal{C}_{n,r}\) consisting of
the orbits of those quadruples in which the matrix \(Y\) has \(n\) distinct
eigenvalues.

Finally, in Section \ref{s:ex} we specialize our machinery to the cases
\(r\leq 3\), making contact with the previously known results for \(r=1,2\).

\section{Non-commutative symplectic geometry}
\label{s:rem}

Let us quickly summarize some definitions and results in non-commutative
symplectic geometry that will be used in the rest of the paper. The interested
reader should look at Ginzburg's lectures \cite{ginz05} for a broad treatment
of the subject.

Let \(A\) denote an associative, not necessarily commutative, algebra over the
commutative ring \(B\). Given an \(A\)-bimodule \(M\), a \emph{derivation}
from \(A\) to \(M\) is a \(A\)-bimodule map \(\app{\theta}{A}{M}\) such that
\[ \theta(ab) = \theta(a)b + a\theta(b) \quad\text{ for all } a,b\in A. \]
We say that \(\theta\) is a derivation \emph{relative to} \(B\) if
\(\theta(b)=0\) for every \(b\in B\), and write \(\Der_{B}(A,M)\) for the
linear space of such derivations. (We will usually write \(\Der_{B}(A)\)
instead of \(\Der_{B}(A,A)\).) The functor
\[ \app{\Der_{B}(A,-)}{\bimod{A}}{\mathbf{Vec}} \]
is representable; we denote its representing object by \(\Omega^{1}_{B}(A)\)
and call it the \(A\)-bimodule of \emph{K\"ahler differentials} of \(A\)
relative to \(B\). It comes naturally equipped with a derivation
\(\app{\de}{A}{\Omega^{1}_{B}(A)}\) relative to \(B\). The tensor algebra of
this \(A\)-bimodule,
\[ \Omega^{\bullet}_{B}(A)\deq \bm{T}_{A}(\Omega^{1}_{B}(A)), \]
turns out to be isomorphic to the \emph{universal differential envelope} of
\(A\) relative to \(B\) (see \cite{cq95}, \cite[Theorem 10.7.1]{ginz05}). As
such, it is naturally a differential graded algebra, with a differential
\(\app{\de}{\Omega^{\bullet}_{B}(A)}{\Omega^{\bullet+1}_{B}(A)}\) extending
the derivation \(\app{\de}{A}{\Omega^{1}_{B}(A)}\) described above. 

More concretely, if we denote by \(\ol{A}\) the quotient of \(A\) by its
subalgebra \(B\) we have the isomorphism
\[ \Omega^{n}_{B}(A)\isom A\otimes_{B} \underbrace{\ol{A} \otimes_{B} \dots
  \otimes_{B} \ol{A}}_{n\text{ times}}. \]
Multiplication is given by the rule
\[ (a_{0}\otimes \dots \otimes a_{n})(a_{n+1}\otimes  \dots\otimes a_{m}) =
\sum_{i=0}^{n} (-1)^{n-i} a_{0}\otimes \dots\otimes a_{i-1}\otimes
a_{i}a_{i+1}\otimes a_{i+2}\otimes \dots\otimes a_{m} \]
and the differential acts as
\[ \de(a_{0}\otimes \dots\otimes a_{n}) = 1\otimes a_{0}\otimes \dots\otimes
a_{n}. \]
In the sequel we are going to denote a generic element \(a_{0}\otimes
a_{1}\otimes \dots\otimes a_{n}\in \Omega^{n}_{B}(A)\) more concisely as
\(a_{0}\de a_{1}\dots \de a_{n}\).

The \emph{non-commutative Karoubi-de Rham complex of} \(A\) relative to \(B\),
denoted \(\DR^{\bullet}_{B}(A)\), is the graded vector space over \(\bb{C}\)
whose degree \(n\) part is defined by
\[ \DR^{n}_{B}(A)\deq \frac{\Omega^{n}_{B}(A)}{\sum_{i=0}^{n}
  \comm{\Omega^{i}_{B}(A)}{\Omega^{n-i}_{B}(A)}}, \]
where \(\comm{\Omega^{i}_{B}(A)}{\Omega^{n-i}_{B}(A)}\) denotes the linear
subspace in \(\Omega^{n}_{B}(A)\) generated by all the graded commutators
\(\comm{a}{b}\) for \(a\in \Omega^{i}_{B}(A)\), \(b\in \Omega^{n-i}_{B}(A)\).
Elements of the linear space \(\DR^{n}_{B}(A)\) will be called the
\emph{Karoubi-de Rham} \(n\)\emph{-forms} of \(A\) relative to \(B\). In
particular for \(n=0\) we have that
\[ \DR^{0}_{B}(A) = \frac{A}{\comm{A}{A}} \]
does not depend on \(B\), hence we will usually denote it simply as
\(\DR^{0}(A)\). Its elements are to be thought of as the ``regular functions''
on the non-commutative space determined by \(A\). We also define
\[ \ol{\DR}^{0}(A)\deq \DR^{0}(A)/B\isom \frac{A}{B + \comm{A}{A}}, \]
in agreement with previous notation.

The differential on \(\Omega^{\bullet}_{B}(A)\) descends to a well defined map
\(\app{\de}{\DR^{\bullet}_{B}(A)}{\DR^{\bullet+1}_{B}(A)}\), making
\(\DR^{\bullet}_{B}(A)\) a differential graded vector space. In fact, one can
introduce a whole ``Cartan calculus'' on \(\DR^{\bullet}_{B}(A)\), that is for
every derivation \(\theta\in \Der_{B}(A)\) we have a degree \(-1\) ``interior
product'' \(\app{i_{\theta}}{\DR^{\bullet}_{B}(A)}{\DR^{\bullet-1}_{B}(A)}\)
and a degree \(0\) ``Lie derivative'' \(\app{\mathcal{L}_{\theta}}
{\DR^{\bullet}_{B}(A)}{\DR^{\bullet}_{B}(A)}\) satisfying all the familiar
relationships from commutative differential geometry \cite[Section
11]{ginz05}.

In this framework, a \emph{non-commutative symplectic manifold} is defined to
be a pair \((A,\omega)\) consisting of an associative algebra \(A\) and a
Karoubi-de Rham 2-form \(\omega\in \DR^{2}_{B}(A)\) which is closed
(\(\de\omega = 0\) in \(\DR^{3}_{B}(A)\)) and non-degenerate, meaning that the
map \(\app{\tilde{\omega}}{\Der_{B}(A)}{\DR^{1}_{B}(A)}\) defined by
\(\theta\mapsto i_{\theta}(\omega)\) is a bijection. In analogy with the
commutative case, we say that a derivation \(\theta\) is \emph{symplectic} if
\(\mathcal{L}_{\theta}(\omega) = 0\), and denote by \(\Der_{B,\omega}(A)\) the
Lie subalgebra of \(\Der_{B}(A)\) consisting of symplectic derivations. It
follows from the Cartan homotopy formula that a derivation \(\theta\) is
symplectic if and only if the 1-form \(\tilde{\omega}(\theta)\) is closed.

There is also no problem in defining a map \(\app{\theta}{\DR^{0}(A)}
{\Der_{B,\omega}(A)}\) sending each 0-form \(f\) to the corresponding
``Hamiltonian derivation'' \(\theta_{f}\deq \tilde{\omega}^{-1}(\de f)\). Just
as in the commutative case, this map fits into an exact sequence of Lie
algebras
\begin{equation}
  \label{eq:26}
  0\rightarrow H^{0}(\DR^{\bullet}_{B}(A))\rightarrow \DR^{0}(A)
  \stackrel{\theta}{\rightarrow} \Der_{B,\omega}(A)\rightarrow
  H^{1}(\DR^{\bullet}_{B}(A))\rightarrow 0
\end{equation}
where the Lie algebra structure on \(\DR^{0}(A)\) is given by the Poisson
brackets determined by \(\omega\), defined e.g. by \(\{f,g\}\deq
i_{\theta_{f}}i_{\theta_{g}}(\omega)\). See \cite[Section 14]{ginz05} for the
details.

Now let us specialize to the case when \(A = \bb{C}\ol{Q}\) is the path
algebra of the double of some quiver \(Q\). Let us denote again with \(I =
\{1, \dots, n\}\) the vertices of \(Q\), and by \(e_{i}\) the trivial (i.e.
length zero) path at the vertex \(i\). Then, as noted in \cite{ginz01,blb02},
\(\bb{C}\ol{Q}\) is most naturally viewed as an algebra over the (commutative)
ring
\[ \bb{C}^{n} = \bb{C}e_{1}\oplus \dots\oplus \bb{C}e_{n} \]
generated by the complete set of mutually orthogonal idempotents \(e_{1},
\dots, e_{n}\). Accordingly, in the sequel we will only consider derivations
and differential forms on \(\bb{C}\ol{Q}\) relative to this subalgebra,
without explicitly displaying this fact in the notation.

A basis for the vector space \(\DR^{0}(\bb{C}\ol{Q})\) is given by the
\emph{necklace words} in \(\bb{C}\ol{Q}\), that is oriented cycles in
\(\ol{Q}\) modulo cyclic permutations. A crucial result \cite{ginz01,blb02}
is that the relative Karoubi-de Rham complex associated to \(\bb{C}\ol{Q}\) is
acyclic:
\begin{equation}
  \label{eq:24}
  H^{k}(\DR^{\bullet}(\bb{C}\ol{Q})) =
  \begin{cases}
    \bb{C}^{n} &\text{ for } k=0\\
    0 &\text{ for } k\geq 1.
  \end{cases}
\end{equation}
The \emph{non-commutative symplectic form} of the quiver \(Q\) is the
Karoubi-de Rham 2-form on \(\bb{C}\ol{Q}\) represented by
\begin{equation}
  \label{eq:omega-nc}
  \omega\deq \sum_{\xi\in Q} \de \xi \de \xi^{*}
\end{equation}
where the sum runs over all the arrows in \(Q\). One can easily verify that
\(\omega\) is both closed and non-degenerate, so that the pair
\((\bb{C}\ol{Q}, \omega)\) is in fact a non-commutative symplectic manifold.

The Poisson brackets induced by \(\omega\) on \(\DR^{0}(\bb{C}\ol{Q})\) can be
given a very explicit expression, as follows. First, let us introduce for
every arrow \(\xi\in \ol{Q}\) a corresponding ``necklace derivative'' operator
\[ \app{\frac{\partial}{\partial \xi}}{\DR^{0}(\bb{C}\ol{Q})}{\bb{C}\ol{Q}} \]
defined on an arrow \(\eta\) of \(\ol{Q}\) as
\[ \frac{\partial \eta}{\partial \xi} =
\begin{cases}
  1 &\text{ if } \eta=\xi\\
  0 &\text{ otherwise}
\end{cases} \]
and extended on a generic necklace word \(w = \eta_{1}\dots \eta_{\ell}\in
\DR^{0}(\bb{C}\ol{Q})\) by
\[ \frac{\partial w}{\partial \xi} = \sum_{k=1}^{\ell} \frac{\partial
  \eta_{k}}{\partial \xi}\eta_{k+1}\dots \eta_{\ell}\eta_{1}\dots \eta_{k-1}. \]
Then the Lie brackets induced on \(\DR^{0}(\bb{C}\ol{Q})\) by the
non-commutative symplectic form \eqref{eq:omega-nc} are given by
\begin{equation}
  \label{eq:25}
  \{w_{1},w_{2}\} = \sum_{\xi\in Q} \left( \frac{\partial w_{1}}{\partial \xi}
    \frac{\partial w_{2}}{\partial \xi^{*}} - \frac{\partial w_{1}}{\partial
      \xi^{*}} \frac{\partial w_{2}}{\partial \xi}\right) \mod
  \comm{\bb{C}\ol{Q}}{\bb{C}\ol{Q}}.
\end{equation}
We also note that, as a consequence of \eqref{eq:24}, the sequence
\eqref{eq:26} becomes in this case simply
\[ 0\rightarrow \bb{C}^{n}\rightarrow \DR^{0}(\bb{C}\ol{Q})\rightarrow
\Der_{\omega}(\bb{C}\ol{Q})\rightarrow 0. \]
Finally, let us clarify what it means for an algebra automorphism of
\(\bb{C}\ol{Q}\) to be symplectic. As it is shown in \cite[\S 10.4]{ginz05},
for every associative algebra \(A\) there is a well-defined map
\[ \app{b}{\DR^{1}(A)}{\comm{A}{A}} \]
given by \(a_{0}\de a_{1}\mapsto \comm{a_{0}}{a_{1}}\). Suppose now that \(A =
\bb{C}\ol{Q}\) for some quiver \(Q\), and let \(\omega\) be any closed
Karoubi-de Rham 2-form on \(\bb{C}\ol{Q}\). By the acyclicity of
\(\DR^{\bullet}(\bb{C}\ol{Q})\), there exists \(\theta\in
\DR^{1}(\bb{C}\ol{Q})\) such that \(\omega = d\theta\); then we can define a
map\footnote{In the ``absolute'' case, that is when the quiver \(Q\) has only
  one vertex (so that \(\bb{C}\ol{Q}\) is a free algebra over \(\bb{C}\)),
  this map is in fact an isomorphism \cite[Proposition 11.5.4(ii)]{ginz05}. We
  do not know whether this result holds also in the ``relative'' case, i.e.
  for a quiver with more than one vertex.}
\begin{equation}
  \label{eq:27}
  \DR^{2}(\bb{C}\ol{Q})_{\text{closed}}\to \comm{\bb{C}\ol{Q}}{\bb{C}\ol{Q}}
\end{equation}
by sending \(\omega\) to \(b(\theta)\in \comm{\bb{C}\ol{Q}}{\bb{C}\ol{Q}}\).
In particular, the non-commutative symplectic form \eqref{eq:omega-nc} is the
differential of the ``non-commutative Liouville 1-form''
\[ \theta = \sum_{\xi\in Q} \xi \de \xi^{*}. \]
The image of this 1-form under the map \eqref{eq:27},
\[ c\deq \sum_{\xi\in Q} \comm{\xi}{\xi^{*}}, \]
will be called, following \cite{blb02}, the \emph{moment element} of the path
algebra \(\bb{C}\ol{Q}\). Then we say that an algebra automorphism
\(\app{\psi}{\bb{C}\ol{Q}}{\bb{C}\ol{Q}}\) is \emph{symplectic} if it
preserves the moment element: \(\psi(c) = c\).

\section{The path algebra of $\Q_{r}$}
\label{s:gen}

The goal of this section is to study the structure of the path algebra of the
quiver \eqref{eq:Qbar} and its automorphism group.

We start by introducing some notation. As in section \ref{s:rem}, we denote by
\(e_{1}\) and \(e_{2}\) the trivial paths at the two vertices \(1\) and \(2\)
in \(\Q_{r}\), and consider \(\bb{C}\Q_{r}\) as an associative algebra over
the ring \(\bb{C}^{2} = \bb{C}e_{1}\oplus \bb{C}e_{2}\) with unit
\(e_{1}+e_{2}\). Let us also define
\begin{equation}
  \label{eq:4}
  \mathcal{A}_{ij}\deq e_{i}\bb{C}\Q_{r}e_{j} \qquad (i,j=1,2)
\end{equation}
as the linear subspace in \(\bb{C}\Q_{r}\) spanned by the paths \(j\to i\).
For the sake of brevity, we put \(\mathcal{A}_{i}\deq \mathcal{A}_{ii}\).
Clearly, as a linear space \(\bb{C}\Q_{r}\) can be decomposed as the direct
sum
\begin{equation}
  \label{eq:1}
  \bb{C}\Q_{r} = \mathcal{A}_{1}\oplus \mathcal{A}_{12}\oplus
  \mathcal{A}_{21}\oplus \mathcal{A}_{2}.
\end{equation}
Let us denote by \(a\) and \(a^{*}\) the two loops at \(1\), by \(b_{1},
\dots, b_{r}\) the \(r\) arrows \(1\rightarrow 2\) and by \(d_{1}, \dots,
d_{r}\) the \(r\) arrows \(1\leftarrow 2\) in \(\Q_{r}\). The letters \(b\)
and \(d\) are meant to resemble, respectively, something pointing to the right
(for arrows \(1\to 2\)) and something pointing to the left (for arrows
\(1\leftarrow 2\)). It is convenient to introduce the matrices
\[ D\deq \begin{pmatrix}
  d_{1}\\
  \vdots\\
  d_{r}
\end{pmatrix},
\quad
B\deq 
\begin{pmatrix}
  b_{1} & \hdots & b_{n}
\end{pmatrix}
\quad\text{ and }\quad
E\deq D\cdot B. \]
We adopt the convention that Greek indices (\(\alpha\), \(\beta\), etc.)
always run from \(1\) to \(r\). Notice that each entry \(e_{\alpha\beta} =
d_{\alpha}b_{\beta}\) of the matrix \(E\) is an element of
\(\mathcal{A}_{1}\), i.e. a cycle at the vertex \(1\) in \(\Q_{r}\).

We can now make a number of trivial observations.
\begin{lem}
  The subspace \(\mathcal{A}_{1}\) is closed under products and is isomorphic
  to the free associative algebra on the \(r^{2}+2\) generators \(a\),
  \(a^{*}\) and \((e_{\alpha\beta})_{\alpha,\beta=1\dots r}\), with unit \(e_{1}\).
\end{lem}
\begin{proof}
  This follows immediately from the definition of the path algebra.
\end{proof}
Of course it is also true that \(\mathcal{A}_{2}\) is a free subalgebra on the
\(r^{2}\) generators \((b_{\alpha}d_{\beta})_{\alpha,\beta=1\dots r}\).
\begin{lem}
  \label{lem:A12-free}
  The subspace \(\mathcal{A}_{12}\) is a free left \(\mathcal{A}_{1}\)-module
  with basis \((d_{1}, \dots, d_{r})\).
\end{lem}
\begin{proof}
  Suppose we are given a relation
  \[ \rho_{1} d_{1} + \dots + \rho_{r} d_{r} = 0 \]
  for some coefficients \((\rho_{\alpha})_{\alpha=1\dots r}\) in
  \(\mathcal{A}_{1}\). By multiplying from the right e.g. by \(b_{1}\), we get
  the equality
  \[ \rho_{1} e_{11} + \dots + \rho_{r} e_{r1} = 0. \]
  As \(\mathcal{A}_{1}\) is a free algebra, this implies \(\rho_{\alpha} = 0\)
  for every \(\alpha=1\dots r\).
\end{proof}
\begin{lem}
  \label{lem:A21-free}
  The subspace \(\mathcal{A}_{21}\) is a free right \(\mathcal{A}_{1}\)-module
  with basis \((b_{1}, \dots, b_{r})\).
\end{lem}
\begin{proof}
  Completely analogous to the previous one.
\end{proof}
Let us also note that, as vector spaces,
\[ \DR^{0}(\bb{C}\Q_{r})\isom \DR^{0}(\mathcal{A}_{1}) \oplus \bb{C}e_{2} \]
because in \(\Q_{r}\) every cycle of nonzero length based at \(2\) can be
turned into a cycle based at \(1\) by a single cyclic shift. In other words,
apart from \(e_{2}\) every ``regular function'' on the non-commutative
manifold \(\bb{C}\Q_{r}\) comes from a ``regular function'' on the
\((r^{2}+2)\)-dimensional non-commutative affine space \(\mathcal{A}_{1}\).

We proceed to study the group of \(\bb{C}^{2}\)-linear automorphisms of
\(\bb{C}\Q_{r}\), denoted simply by \(\Aut \bb{C}\Q_{r}\). Its elements are
morphisms of algebras \(\bb{C}\Q_{r}\to \bb{C}\Q_{r}\) which are invertible
and fix both \(e_{1}\) and \(e_{2}\). It follows that each \(\psi\in \Aut
\bb{C}\Q_{r}\) preserves the decomposition \eqref{eq:1}, since for every path
\(p\in \bb{C}\Q_{r}\)
\[ \psi(e_{i}pe_{j}) = e_{i}\psi(p)e_{j} \]
so that \(p\in \mathcal{A}_{ij}\) implies \(\psi(p)\in \mathcal{A}_{ij}\).
This means in particular that every \(\psi\in \Aut \bb{C}\Q_{r}\)
automatically induces, by restriction, three bijections
\begin{equation}
  \label{eq:def-res}
  \app{\psi_{1}}{\mathcal{A}_{1}}{\mathcal{A}_{1}}, \quad
  \app{\psi_{12}}{\mathcal{A}_{12}}{\mathcal{A}_{12}} \quad\text{ and }\quad
  \app{\psi_{21}}{\mathcal{A}_{21}}{\mathcal{A}_{21}}.
\end{equation}
\begin{lem}
  The map
  \begin{equation}
    \label{eq:res1}
    \app{\res_{1}}{\Aut \bb{C}\Q_{r}}{\Aut \mathcal{A}_{1}}
  \end{equation}
  defined by \(\psi\mapsto \psi_{1}\) is a morphism of groups.
\end{lem}
\begin{proof}
  Immediate from the definitions.
\end{proof}
To clarify the nature of the other two restriction maps we need the following
concept from group theory. Suppose we are given two groups \(G\) and \(H\) and
an action \(\app{\varphi}{G}{\Aut H}\) of \(G\) on \(H\). A map
\(\app{f}{G}{H}\) is then called a \emph{crossed (homo)morphism of groups} if
\[ f(ab) = f(a) \varphi_{a}(f(b)) \quad\text{ for all } a,b\in G. \]
When the action \(\varphi\) is trivial, a crossed morphism is just an ordinary
morphism of groups.

In the present situation, let us take \(G = \Aut \bb{C}\Q_{r}\) and \(H =
\GL_{r}(\mathcal{A}_{1})\), the group of invertible \(r\times r\) matrices
with entries in the free algebra \(\mathcal{A}_{1}\). The morphism
\eqref{eq:res1} induces an action of \(G\) on \(H\) obtained by letting
\(\psi_{1} = \res_{1}(\psi)\) act on each matrix element:
\begin{equation}
  \label{eq:act-Aut-GL}
  \psi(M)\deq (\psi_{1}(M_{\alpha\beta}))_{\alpha,\beta=1\dots r} \quad\text{
    for all } M\in \GL_{r}(\mathcal{A}_{1}).
\end{equation}
We claim that the two restriction maps defined by \(\psi\mapsto \psi_{12}\)
and \(\psi\mapsto \psi_{21}\) induce two crossed morphisms \(G\to H\) with
respect to the above action.

To see this, notice first that by lemma \ref{lem:A12-free} the map
\(\psi_{12}\) is completely specified by its values on the basis \((d_{1},
\dots, d_{r})\), and these values in turn can be expressed on this basis in a
unique way; in other words, there exists a unique \(r\times r\) matrix
\(M^{\psi}\) with entries in \(\mathcal{A}_{1}\) such that
\[ \psi_{12}(d_{\alpha}) = \sum_{\beta=1}^{r} M^{\psi}_{\alpha\beta} d_{\beta}. \]
Similarly, by lemma \ref{lem:A21-free}, the map \(\psi_{21}\) is also
completely specified by a unique \(r\times r\) matrix \(N^{\psi}\) with
entries in \(\mathcal{A}_{1}\) such that
\[ \psi_{21}(b_{\alpha}) = \sum_{\beta=1}^{r} b_{\beta} N^{\psi}_{\beta\alpha}. \]
\begin{teo}
  \label{teo:cr-mor}
  The maps
  \begin{equation}
    \label{eq:cr-mor}
    \app{M}{\Aut \bb{C}\Q_{r}}{\GL_{r}(\mathcal{A}_{1})^{\op}} \quad\text{ and
    }\quad \app{N}{\Aut \bb{C}\Q_{r}}{\GL_{r}(\mathcal{A}_{1})}
  \end{equation}
  defined, respectively, by \(\psi\mapsto M^{\psi}\) and \(\psi\mapsto
  N^{\psi}\) are crossed morphism of groups with respect to the action
  \eqref{eq:act-Aut-GL}.
\end{teo}
\begin{proof}
  Take \(\psi,\sigma \in \Aut \bb{C}\Q_{r}\). The action of \(\psi\) followed
  by \(\sigma\) on \(d_{\alpha}\) is
  \[ \sigma(\psi(d_{\alpha})) = \sigma(\sum_{\beta=1}^{r}
  M^{\psi}_{\alpha\beta} d_{\beta}) = \sum_{\beta=1}^{r}
  \sigma(M^{\psi}_{\alpha\beta}) \sigma_{12}(d_{\beta}) = \sum_{\beta=1}^{r}
  \sum_{\gamma=1}^{r} \sigma(M^{\psi})_{\alpha\beta} M^{\sigma}_{\beta\gamma}
  d_{\gamma} \]
  where \(\sigma(M^{\psi})\) is given exactly by \eqref{eq:act-Aut-GL}. In
  other words,
  \[ \sigma_{12}(\psi_{12}(d_{\alpha})) = \sum_{\gamma=1}^{r}
  (\sigma(M^{\psi}) M^{\sigma})_{\alpha\gamma} d_{\gamma} \]
  hence the matrix corresponding to \((\sigma\circ \psi)_{12}\) is
  \(\sigma(M^{\psi}) M^{\sigma}\). This would be the same as saying that the
  map defined by \(\psi\mapsto M^{\psi}\) is a crossed morphism from \(\Aut
  \bb{C}\Q_{r}\) to the opposite of \(\GL_{r}(\mathcal{A}_{1})\), except that
  we did not prove yet that \(M^{\psi}\) is always invertible. But applying
  the above result to the equalities \(\psi^{-1}\circ \psi = \psi\circ
  \psi^{-1} = \id_{\bb{C}\Q_{r}}\) we get that
  \[ \psi^{-1}(M^{\psi}) M^{\psi^{-1}} = I_{r} \quad\text{ and }\quad
  \psi(M^{\psi^{-1}}) M^{\psi} = I_{r}. \]
  From the second equality it follows that \(\psi(M^{\psi^{-1}})\) is a left
  inverse for \(M^{\psi}\), and acting with \(\psi\) on the first equality
  (notice that \(\psi(I_{r}) = I_{r}\) since \(\psi\) fixes \(e_{1}\)) we get
  that \(M^{\psi} \psi(M^{\psi^{-1}}) = I_{r}\), hence \(\psi(M^{\psi^{-1}})\)
  is also a right inverse for \(M^{\psi}\). This concludes the proof for the
  map \(M\).

  With regard to \(N\), we only need to perform the analogous computation for
  \(\sigma(\psi(b_{\alpha}))\) to find out that the matrix corresponding to
  \((\sigma\circ \psi)_{21}\) is given by \(N^{\sigma} \sigma(N^{\psi})\). By
  applying this result to the equalities \(\psi^{-1}\circ \psi = \psi\circ
  \psi^{-1} = \id_{\bb{C}\Q_{r}}\) we can again deduce that \(N^{\psi}\) is
  always invertible, and its inverse coincides with \(\psi(N^{\psi^{-1}})\).
\end{proof}
Let us study the structure of the group \(\Aut \bb{C}\Q_{r}\) in more detail,
following \cite[Section 6.1]{bp08}. Let \(\mathcal{I}_{r}\) denote the
two-sided ideal in \(\bb{C}\Q_{r}\) generated by \(e_{2}\), \(d_{1}, \dots,
d_{r}\) and \(b_{1}, \dots, b_{r}\). Clearly
\[ \bb{C}\Q_{r}/\mathcal{I}_{r} \isom \falg{a,a^{*}} = \bb{C}\Q_{0} \]
where \(\Q_{0}\) is the quiver with a single vertex and two loops \(a\) and
\(a^{*}\) on it. 

Every \(\bb{C}^{2}\)-linear map \(\bb{C}\Q_{r}\to \bb{C}\Q_{r}\) preserves the
ideal \(\mathcal{I}_{r}\), hence every \(\psi\in \Aut \bb{C}\Q_{r}\) descends
to an automorphism of \(\bb{C}\Q_{0}\) and this gives a map
\[ \app{\pi}{\Aut \bb{C}\Q_{r}}{\Aut \bb{C}\Q_{0}} \]
which is a morphism of groups. On the other hand, every automorphism of
\(\bb{C}\Q_{0}\) extends to an automorphism of \(\bb{C}\Q_{r}\) acting as the
identity on the additional arrows, hence we have a split extension of groups,
i.e. a semidirect product
\begin{equation}
  \label{eq:sdp-Aut}
  \Aut \bb{C}\Q_{r} = K\rtimes \Aut \bb{C}\Q_{0}
\end{equation}
where \(K\) is the kernel of \(\pi\), that is the (normal) subgroup of
automorphisms of \(\bb{C}\Q_{r}\) which fix \(a\) and \(a^{*}\) after killing
all the other arrows. For lack of a better word, we shall call such an
automorphism \emph{reduced}.
\begin{defn}
  The subgroup \(K\) of \(\Aut \bb{C}\Q_{r}\) defined by the splitting
  \eqref{eq:sdp-Aut} is called the group of \textbf{reduced automorphisms} of
  \(\bb{C}\Q_{r}\).
\end{defn}
In the sequel we will be particularly interested in \emph{symplectic}
automorphisms, that is automorphisms of \(\bb{C}\Q_{r}\) fixing its moment
element. Quite generally, given a quiver \(Q\) and a path \(p\in \bb{C}Q\) we
can define
\begin{equation}
  \label{eq:def-aut-fix}
  \Aut(\bb{C}Q;p)\deq \set{\psi\in \Aut \bb{C}Q | \psi(p)=p}
\end{equation}
as the subgroup of automorphisms of \(\bb{C}Q\) fixing \(p\). Now suppose we
take some path \(c_{r}\in \bb{C}\Q_{r}\) and let \(c_{0}\) be the image of
\(c_{r}\) in \(\bb{C}\Q_{0}\) modulo \(\mathcal{I}_{r}\). Then we can repeat
the above argument to show that the group \(\Aut (\bb{C}\Q_{r};c_{r})\) of
\(c_{r}\)-preserving automorphisms of \(\bb{C}\Q_{r}\) is again a semidirect
product
\begin{equation}
  \label{eq:sdp-simp}
  \Aut (\bb{C}\Q_{r};c_{r}) = K_{c_{r}}\rtimes \Aut (\bb{C}\Q_{0};c_{0})
\end{equation}
where now \(K_{c_{r}}\) is the subgroup of reduced automorphisms fixing
\(c_{r}\).

\section{The non-commutative symplectic geometry of zigzag quivers}
\label{s:nc}

Let us denote by \((Z_{r})_{r\geq 1}\) the family of quivers defined in the
following way. Each quiver has two vertices, call them \(1\) and \(2\), and a
loop \(\app{a}{1}{1}\). Moreover, the quiver \(Z_{r}\) has:
\begin{itemize}
\item an arrow \(\app{x_{i}}{2}{1}\) for each \(i = 1, \dots,
  \ceil{\frac{r}{2}}\), and
\item an arrow \(\app{y_{j}}{1}{2}\) for each \(j = 1, \dots,
  \floor{\frac{r}{2}}\).
\end{itemize}
Here \(\floor{z}\) and \(\ceil{z}\) denote, respectively, the largest integer
not greater than \(z\) and the smallest integer not less than \(z\), as usual.
It follows that the double \(\ol{Z}_{r}\) has an additional loop
\(\app{a^{*}}{1}{1}\), \(\ceil{\frac{r}{2}}\) additional arrows
\(\app{x_{i}^{*}}{1}{2}\) and \(\floor{\frac{r}{2}}\) additional arrows
\(\app{y_{j}^{*}}{2}{1}\). It is immediate to check that \(\ol{Z}_{r}\)
coincides with the quiver \eqref{eq:Qbar}. We will identify the generators of
\(\bb{C}\Q_{r}\) (as labeled in the previous section) with the generators in
\(\bb{C}\ol{Z}_{r}\) in the following way:
\[ d_{\alpha} = 
\begin{cases}
  -x_{(\alpha+1)/2} &\text{ for } \alpha \text{ odd}\\
  y^{*}_{\alpha/2} &\text{ for } \alpha \text{ even}
\end{cases}
\]
\[ b_{\alpha} =
\begin{cases}
  x^{*}_{(\alpha+1)/2} &\text{ for } \alpha \text{ odd}\\
  y_{\alpha/2} &\text{ for } \alpha \text{ even.}
\end{cases}
\]
In terms of this new notation, the matrices \(D\), \(B\) and \(E\) read
\begin{equation}
  \label{eq:7}
  D = \begin{pmatrix}
    -x_{1}\\
    y_{1}^{*}\\
    -x_{2}\\
    y_{2}^{*}\\
    \vdots
  \end{pmatrix},
  \qquad
  B =
  \begin{pmatrix}
    x_{1}^{*} & y_{1} & x_{2}^{*} & y_{2} & \hdots
  \end{pmatrix},
\end{equation}
\begin{equation}
  \label{eq:8}
  E = 
  \begin{pmatrix}
    -x_{1}x_{1}^{*} & -x_{1}y_{1} & -x_{1}x_{2}^{*} & -x_{1}y_{2} & \hdots\\
    y_{1}^{*}x_{1}^{*} & y_{1}^{*}y_{1} & y_{1}^{*}x_{2}^{*} & y_{1}^{*}y_{2} & \hdots\\
    -x_{2}x_{1}^{*} & -x_{2}y_{1} & -x_{2}x_{2}^{*} & -x_{2}y_{2} & \hdots\\
    y_{2}^{*}x_{1}^{*} & y_{2}^{*}y_{1} & y_{2}^{*}x_{2}^{*} & y_{2}^{*}y_{2} & \hdots\\
    \vdots & \vdots & \vdots & \vdots & \ddots
  \end{pmatrix}.
\end{equation}
For later use, we define the sequence \((q_{r})_{r\geq 1}\) by \(q_{r}\deq
\ceil{r/2} \floor{r/2}\); these are usually known as \emph{quarter-square
  numbers}. Clearly, \(q_{r}\) coincides with the number of different cycles at \(1\) that
may be obtained by composing one of the arrows \(x_{1}, \dots,
x_{\ceil{r/2}}\) with one of the arrows \(y_{1}, \dots, y_{\floor{r/2}}\) in
the quiver \(Z_{r}\).
\begin{rem}
  Although it will not play any r\^ole in the sequel, we note that the Poisson
  brackets \eqref{eq:25} naturally defined on \(\DR^{0}(\bb{C}\Q_{r})\)
  induce, by restriction to the subalgebra \(\mathcal{A}_{1}\), a Lie algebra
  structure on the linear space \(\DR^{0}(\mathcal{A}_{1})\) whose nonzero
  part reads
  \[ \{a,a^{*}\} = 1 \quad\text{ and }\quad \{e_{\alpha\beta}, e_{\gamma\delta}\}
  = \delta_{\beta\gamma} e_{\alpha\delta} - \delta_{\alpha\delta}
  e_{\gamma\beta}. \]
  In particular the \(e_{\alpha\beta}\) obey the same relations that hold in
  the Lie algebra \(\mathfrak{gl}_{r}(\bb{C})\), as emphasized in \cite[\S
  5.1]{bp08} for the case \(r=2\).
\end{rem}
According to the definition \eqref{eq:omega-nc}, the non-commutative
symplectic form of \(Z_{r}\) is given by
\[ \omega = \de a\, \de a^{*} + \sum_{i=1}^{\ceil{r/2}} \de x_{i} \de x_{i}^{*}
+ \sum_{j=1}^{\floor{r/2}} \de y_{j} \de y_{j}^{*} \]
and the corresponding moment element in \(\bb{C}\ol{Z}_{r} = \bb{C}\Q_{r}\) is
\[ c_{r}\deq [a,a^{*}] + \sum_{i=1}^{\ceil{r/2}} [x_{i},x_{i}^{*}] +
\sum_{j=1}^{\floor{r/2}} [y_{j},y_{j}^{*}]. \]
An element of \(\Aut \bb{C}\Q_{r}\) fixing \(c_{r}\) will be called a
\emph{symplectic automorphism of} \(\bb{C}\Q_{r}\), or a
\emph{symplectomorphism} for short. Using the notation introduced in
\eqref{eq:def-aut-fix}, the group of symplectic automorphisms of \(\Q_{r}\)
will be denoted by \(\Aut(\bb{C}\Q_{r};c_{r})\). 

Notice that the quotient of \(c_{r}\) by the ideal \(\mathcal{I}_{r}\),
\[ c_{0}\deq [a,a^{*}], \]
is exactly the moment element corresponding to the non-commutative symplectic
form on the path algebra \(\bb{C}\Q_{0}\) (with \(\Q_{0}\) seen as a double in
the obvious way). We conclude by \eqref{eq:sdp-simp} that every symplectic
automorphism of \(\bb{C}\Q_{r}\) can be factorized in a unique way as the
composition of a reduced symplectic automorphism followed by a symplectic
automorphism of \(\bb{C}\Q_{0}\) (acting only on \(a\) and \(a^{*}\)). As the
group \(\Aut (\bb{C}\Q_{0};c_{0})\) is ``well-known'' (it is isomorphic to the
automorphism group of the first Weyl algebra, see the proof of lemma
\ref{lem:aut0-tame} below), the problem of understanding the group
\(\Aut(\bb{C}\Q_{r};c_{r})\) largely reduces to understanding its subgroup
\(K_{c_{r}}\) of reduced automorphisms.

Next we note that \(c_{r}\) can be decomposed as \(c_{r}^{1} + c_{r}^{2}\),
where
\[ c_{r}^{1}\deq [a,a^{*}] + \sum_{i=1}^{\ceil{r/2}} x_{i}x_{i}^{*} -
\sum_{j=1}^{\floor{r/2}} y_{j}^{*}y_{j} = [a,a^{*}] - \sum_{\alpha=1}^{r}
d_{\alpha}b_{\alpha} \]
is a cycle based at \(1\) and
\[ c_{r}^{2}\deq - \sum_{i=1}^{\ceil{r/2}} x_{i}^{*}x_{i} +
\sum_{j=1}^{\floor{r/2}} y_{j}y_{j}^{*} = \sum_{\alpha=1}^{r}
b_{\alpha}d_{\alpha} \]
is a cycle based at \(2\). Being \(\bb{C}^{2}\)-linear, every automorphism in
\(\Aut(\bb{C}\Q_{r};c_{r})\) must preserve separately the two cycles
\(c_{r}^{1}\) and \(c_{r}^{2}\); this means that
\[ \Aut(\bb{C}\Q_{r};c_{r}) = \Aut(\bb{C}\Q_{r};c_{r}^{1}) \cap
\Aut(\bb{C}\Q_{r};c_{r}^{2}). \]
The second component of this intersection can be characterized quite
satisfactorily.
\begin{teo}
  \label{teo:aut-fix-cr2}
  An automorphism \(\psi\in \Aut \bb{C}\Q_{r}\) fixes \(c_{r}^{2}\) if and
  only if \((N^{\psi})^{-1} = M^{\psi}\) in \(\GL_{r}(\mathcal{A}_{1})\).
\end{teo}
\begin{proof}
  We have that
  \[ \psi(c_{r}^{2}) = \psi(\sum_{\alpha=1}^{r} b_{\alpha}d_{\alpha}) =
  \sum_{\alpha,\beta,\gamma=1}^{r} b_{\gamma} N^{\psi}_{\gamma\alpha}
  M^{\psi}_{\alpha\beta} d_{\beta}. \]
  Then \(\psi\) fixes \(c_{r}^{2}\) if and only if
  \[ \sum_{\beta,\gamma=1}^{r} b_{\gamma} (N^{\psi}M^{\psi})_{\gamma\beta}
  d_{\beta} = \sum_{\alpha=1}^{r} b_{\alpha}d_{\alpha}, \]
  or in other words
  \[ \sum_{\beta,\gamma=1}^{r} b_{\gamma} (N^{\psi}M^{\psi} - I)_{\gamma\beta}
  d_{\beta} = 0. \]
  Set \(C\deq N^{\psi}M^{\psi} - I\). Multiplying the previous relation by
  \(d_{1}\) from the left and \(b_{1}\) from the right we get
  \[ \sum_{\beta,\gamma=1}^{r} d_{1}b_{\gamma} C_{\gamma\beta} d_{\beta}b_{1}
  = \sum_{\gamma=1}^{r} e_{1\gamma} \sum_{\beta=1}^{r} C_{\gamma\beta}
  e_{\beta 1} = 0. \]
  As \(\mathcal{A}_{1}\) is a free algebra, this implies \(\sum_{\beta}
  C_{\gamma\beta} e_{\beta 1} = 0\) for every \(\gamma\), which in turn
  implies that \(C_{\gamma\beta} = 0\) for every \(\gamma\) and \(\beta\).
  Hence \(C=0\), i.e. \(N^{\psi} = (M^{\psi})^{-1}\), as claimed.
\end{proof}
It follows that for \(c_{r}^{2}\)-preserving automorphisms the two crossed
morphisms \eqref{eq:cr-mor} are related by the isomorphism
\(\app{(\cdot)^{-1}}{\GL_{r}(\mathcal{A}_{1})^{\op}}{\GL_{r}(\mathcal{A}_{1})}\),
i.e. the following diagram commutes:
\[ \xymatrix{
  \Aut (\bb{C}\Q_{r};c_{r}^{2}) \ar[r]^{M} \ar[dr]_{N} &
  \GL_{r}(\mathcal{A}_{1})^{\op} \ar[d]^{(\cdot)^{-1}}\\
  & \GL_{r}(\mathcal{A}_{1})} \]
Unfortunately, the condition of preserving \(c_{r}^{1}\) does not seem to
admit such a straightforward description (even considering only reduced
automorphisms). For this reason the structure of the group
\(\Aut(\bb{C}\Q_{r};c_{r})\) remains somewhat unclear.

To remedy this situation we introduce some classes of ``nice'' automorphisms.
\begin{defn}
  An automorphism \(\psi\in \Aut \bb{C}\Q_{r}\) will be called:
  \begin{itemize}
  \item \textbf{triangular}\footnote{In \cite{bp08} and \cite{mt13}
      automorphisms of this kind are called \emph{strictly triangular},
      reserving the word \emph{triangular} for the larger class of
      automorphisms preserving the subalgebra \(\bb{C}Z_{r}\). As this weaker
      notion will play no r\^ole in the sequel, we decided to drop the
      ``strictly'' part and use the simplest name for the most useful
      concept.} if it fixes each of the arrows \((a, x_{1}, \dots,
    x_{\ceil{r/2}}, y_{1}, \dots, y_{\floor{r/2}})\),
  \item \textbf{op-triangular} if it fixes each of the arrows \((a^{*},
    x_{1}^{*}, \dots, x_{\ceil{r/2}}^{*}, y_{1}^{*}, \dots, y_{\floor{r/2}}^{*})\),
  \item \textbf{affine} if it sends each arrow in \(\Q_{r}\) to a polynomial
    which is \emph{at most linear} in each arrow.
  \end{itemize}
\end{defn}
Notice that the definition of (op-)triangularity depends on the choice of the
unstarred arrows in \(\Q_{r}\).

Our next goal is to characterize the symplectic automorphisms in each of these
classes.

\subsection{Triangular symplectomorphisms}
\label{staut}

Let \(\psi\) be a triangular automorphism of \(\bb{C}\Q_{r}\). Without
loss of generality, we can write the action of \(\psi\) as
\[ \begin{cases}
  a\mapsto a\\
  x_{i}\mapsto x_{i}\\
  y_{j}\mapsto y_{j}
\end{cases}
\qquad
\begin{cases}
  a^{*}\mapsto a^{*} + h\\
  x_{i}^{*}\mapsto x_{i}^{*} + s_{i}\\
  y_{j}^{*}\mapsto y_{j}^{*} + t_{j}
\end{cases} \]
for some \(h\in \mathcal{A}_{1}\), \(s_{1}, \dots, s_{\ceil{r/2}}\in
\mathcal{A}_{21}\) and \(t_{1}, \dots, t_{\floor{r/2}}\in \mathcal{A}_{12}\).
Now \(\psi\) will preserve \(c_{r}^{2}\) if and only if
\begin{equation}
  \label{eq:9}
  \sum_{i} s_{i}x_{i} = \sum_{j} y_{j}t_{j}.
\end{equation}
This equality can only hold when each \(s_{i}\) can be written as \(\sum_{j}
y_{j}u_{ij}\) for some coefficients \(u_{ij}\in \mathcal{A}_{1}\); then
equation \eqref{eq:9} becomes
\[ \sum_{j} y_{j} (\sum_{i} u_{ij}x_{i} - t_{j}) = 0, \]
which implies that \(t_{j} = \sum_{i} u_{ij}x_{i}\) for every \(j = 1, \dots,
\floor{r/2}\). Hence \(\psi\) acts as
\begin{equation}
  \label{eq:10}
  \begin{cases}
    a^{*}\mapsto a^{*} + h\\
    x_{i}^{*}\mapsto x_{i}^{*} + \sum_{j} y_{j}u_{ij}\\
    y_{j}^{*}\mapsto y_{j}^{*} + \sum_{i} u_{ij}x_{i}
  \end{cases} \quad\text{ for some } u_{ij}\in \mathcal{A}_{1}.
\end{equation}
Such an automorphism preserves \(c_{r}^{1}\) if and only if
\[ [a,h] + \sum_{i,j} x_{i}y_{j}u_{ij} - \sum_{i,j} u_{ij}x_{i}y_{j} = 0 \]
or, using \(e_{2i-1,2j} = -x_{i}y_{j}\),
\begin{equation}
  \label{eq:sts-aut}
  [a,h] - \sum_{i,j} [e_{2i-1,2j}, u_{ij}] = 0.
\end{equation}
Thus we are led to study equations in \(\mathcal{A}_{1}\) of the form
\begin{equation}
  \label{eq:char-eq}
  \sum_{k=0}^{n} [g_{k},u_{k}] = 0
\end{equation}
where \(g_{0}, \dots, g_{n}\) are (distinct) generators of \(\mathcal{A}_{1}\)
and \(u_{0}, \dots, u_{n}\) are unknowns. This problem is completely solved by
the following result.
\begin{teo}
  \label{teo:sol-char-eq}
  The \((n+1)\)-tuple \((u_{0}, \dots, u_{n})\) satisfies equation
  \eqref{eq:char-eq} if and only if there exists \(f\in
  \ol{\DR}^{0}(\falg{g_{0}, \dots, g_{n}})\) such that \(u_{k} =
  \frac{\partial f}{\partial g_{k}}\) for all \(k=0\dots n\).
\end{teo}
The above theorem is known (see e.g. Proposition 1.5.13 in \cite{ginz06});
however, as we were unable to locate a complete proof in the literature, we
supply our own proof in Appendix \ref{s:app}.

It is now easy to explicitly describe all the triangular symplectomorphisms of
\(\bb{C}\Q_{r}\). Let us denote by \(\ST_{c}\) the subgroup of
\(\Aut(\bb{C}\Q_{r};c_{r})\) consisting of such automorphisms. Let also
\(F_{r}\) denote the (free) subalgebra of \(\mathcal{A}_{1}\) generated by the
\(q_{r}+1\) elements \(a\) and \(b_{ij}\deq -e_{2i-1,2j} = x_{i}y_{j}\) for
\(i=1\dots \ceil{r/2}\) and \(j=1\dots \floor{r/2}\). We define a map
\[ \app{\Lambda}{\ol{\DR}^{0}(F_{r})}{\Aut \bb{C}\Q_{r}} \]
by letting a necklace word \(f\in \ol{\DR}^{0}(F_{r})\) correspond to the
triangular automorphism
\[ \begin{cases}
  a\mapsto a\\
  x_{i}\mapsto x_{i}\\
  y_{j}\mapsto y_{j}
\end{cases}
\qquad
\begin{cases}
  a^{*}\mapsto a^{*} + \frac{\partial f}{\partial a}\\
  x_{i}^{*}\mapsto x_{i}^{*} + \sum_{j=1}^{\floor{r/2}} y_{j}
  \frac{\partial f}{\partial b_{ij}}\\
  y_{j}^{*}\mapsto y_{j}^{*} + \sum_{i=1}^{\ceil{r/2}}
  \frac{\partial f}{\partial b_{ij}} x_{i}
\end{cases} \]
\begin{teo}
  \label{teo:iso-tri}
  The map \(\Lambda\) is an isomorphism between the abelian group
  \(\p{\ol{\DR}^{0}(F_{r}),+}\) and \(\ST_{c}\).
\end{teo}
\begin{proof}
  We saw that every element of \(\ST_{c}\) has the form \eqref{eq:10} and is
  determined by a \((q_{r}+1)\)-tuple \((h,u_{ij})\) of elements of
  \(\mathcal{A}_{1}\) that satisfies equation \eqref{eq:sts-aut}. By theorem
  \ref{teo:sol-char-eq}, all those tuples are obtained by taking
  \(h=\frac{\partial f}{\partial a}\) and \(u_{ij} = \frac{\partial
    f}{\partial b_{ij}}\) for some \(f\in \ol{\DR}^{0}(F_{r})\), which gives
  exactly \(\Lambda(f)\) as defined above. It follows that \(\Lambda\) is a
  surjective map \(\ol{\DR}^{0}(F_{r})\to \ST_{c}\). It is also injective,
  since necklace derivatives have no kernel on \(\ol{\DR}^{0}(F_{r})\).
  Finally, it is a morphism of (abelian) groups:
  \[ \Lambda(f_{1}+f_{2}) = \Lambda(f_{1})\circ \Lambda(f_{2}). \]
  This follows easily by noting that necklace derivatives are
  \(\bb{C}\)-linear and \(\Lambda(f)\) translates each starred arrow by a
  non-commutative polynomial in \(a\) and \(b_{ij}=x_{i}y_{j}\), and every
  automorphism in \(\ST_{c}\) fixes such an element.
\end{proof}
When \(r\) is even, it is easy to verify that the matrix in
\(\GL_{r}(\mathcal{A}_{1})\) associated to a generic element \(\Lambda(f)\in
\ST_{c}\) via the crossed morphism \(N\) is given by
\begin{equation}
  \label{eq:N-Lambda-f}
  N^{\Lambda(f)} = I_{r} +
  \begin{pmatrix}
    \frac{\partial f}{\partial b_{11}} & \dots & \frac{\partial f}{\partial
      b_{\ceil{r/2}1}}\\
    \vdots & \ddots & \vdots\\
    \frac{\partial f}{\partial b_{1\floor{r/2}}} & \dots &
    \frac{\partial f}{\partial b_{\ceil{r/2}\floor{r/2}}}  
  \end{pmatrix}
  \otimes
  \begin{pmatrix}
    0 & 0\\
    1 & 0
  \end{pmatrix} =
  \begin{pmatrix}
    1 & 0 & 0 & 0 & \hdots\\
    \frac{\partial f}{\partial b_{11}} & 1 & \frac{\partial f}{\partial
      b_{21}} & 0 & \hdots\\
    0 & 0 & 1 & 0 & \hdots\\
    \frac{\partial f}{\partial b_{12}} & 0 & \frac{\partial f}{\partial
      b_{22}} & 1 & \hdots\\
    \vdots & \vdots & \vdots & \vdots & \ddots
  \end{pmatrix}.
\end{equation}
When \(r\) is odd one has to take the above matrix for \(r+1\) and remove the
last row and the last column. 

By theorem \ref{teo:aut-fix-cr2}, the matrix \(M^{\Lambda(f)}\) is just the
inverse of \(N^{\Lambda(f)}\); as \(\ST_{c}\) is abelian, to obtain it we
simply need to replace \(f\) with \(-f\) in \eqref{eq:N-Lambda-f}.

We also note that the triangular symplectomorphism \(\Lambda(f)\) is reduced
if and only if every monomial in \(\frac{\partial f}{\partial a}\) depends on
at least one of the variables \(b_{ij}\), since it is exactly in this case
that the translation affecting \(a^{*}\) gets killed after quotienting out by
the ideal \(\mathcal{I}_{r}\).

\subsection{Op-triangular symplectomorphisms}

To describe op-triangular symplectomorphisms we only need to adapt the above
analysis in the obvious way. Writing the action of a generic op-triangular
automorphism as
\[ \begin{cases}
  a\mapsto a + h'\\
  x_{i}\mapsto x_{i} + s'_{i}\\
  y_{j}\mapsto y_{j} + t'_{j}
\end{cases}
\qquad
\begin{cases}
  a^{*}\mapsto a^{*}\\
  x_{i}^{*}\mapsto x_{i}^{*}\\
  y_{j}^{*}\mapsto y_{j}^{*}
\end{cases} \]
for some \(h'\in \mathcal{A}_{1}\), \(s'_{1}, \dots, s'_{\ceil{r/2}}\in
\mathcal{A}_{12}\) and \(t'_{1}, \dots, t'_{\floor{r/2}}\in \mathcal{A}_{21}\)
and requiring \(c_{r}^{2}\) to be fixed we deduce that
\[ s'_{i} = \sum_{j} v_{ij}y_{j}^{*} \quad\text{ and }\quad t'_{j} = \sum_{i} 
x_{i}^{*}v_{ij} \]
for some coefficients \(v_{ij}\in \mathcal{A}_{1}\). An automorphism of this
form will fix \(c_{r}^{1}\) if and only if
\[ [a^{*},h'] + \sum_{i,j} [y_{j}^{*}x_{i}^{*}, v_{ij}] = 0. \]
This is again an equation of the form \eqref{eq:char-eq}, as
\(y_{j}^{*}x_{i}^{*} = e_{2j,2i-1}\). By theorem \ref{teo:sol-char-eq}, its
solutions are parametrized by elements of \(\ol{\DR}^{0}(F_{r}^{*})\), where
\(F_{r}^{*}\) is the free subalgebra of \(\mathcal{A}_{1}\) generated by the
\(q_{r}+1\) elements \(a^{*}\) and \(b_{ij}^{*} = e_{2i,2j-1} =
y_{j}^{*}x_{i}^{*}\), by taking
\[ h' = \frac{\partial f}{\partial a^{*}} \quad\text{ and }\quad v_{ij} =
\frac{\partial f}{\partial b_{ij}^{*}}. \]
By mimicking the proof of theorem \ref{teo:iso-tri} it is straightforward to
show that the map
\[ \app{\Lambda'}{\ol{\DR}^{0}(F_{r}^{*})}{\Aut \bb{C}\ol{Q}_{r}} \]
defined by sending \(f\in \ol{\DR}^{0}(F_{r}^{*})\) to the automorphism
\[ \begin{cases}
  a\mapsto a + \frac{\partial f}{\partial a^{*}}\\
  x_{i}\mapsto x_{i} + \sum_{j} \frac{\partial f}{\partial b_{ij}^{*}} y_{j}^{*}\\
  y_{j}\mapsto y_{j} + \sum_{i} x_{i}^{*} \frac{\partial f}{\partial b_{ij}^{*}}
\end{cases}
\qquad
\begin{cases}
  a^{*}\mapsto a^{*}\\
  x_{i}^{*}\mapsto x_{i}^{*}\\
  y_{j}^{*}\mapsto y_{j}^{*}\\
\end{cases} \]
is an isomorphism between \((\ol{\DR}^{0}(F_{r}^{*}),+)\) and the group of
op-triangular symplectic automorphisms of \(\bb{C}\Q_{r}\), that we will
denote by \(\oST_{c}\). The matrix associated to \(\Lambda'(f)\) by the
crossed morphism \(N\) is
\[ N^{\Lambda'(f)} =
\begin{pmatrix}
  1 & \frac{\partial f}{\partial b^{*}_{11}} & 0 & \frac{\partial f}
  {\partial b^{*}_{21}} & \hdots\\
  0 & 1 & 0 & 0 & \hdots\\
  0 & \frac{\partial f}{\partial b^{*}_{12}} & 1 & \frac{\partial f}
  {\partial b^{*}_{22}} & \hdots\\
  0 & 0 & 0 & 1 & \hdots\\
  \vdots & \vdots & \vdots & \vdots & \ddots
\end{pmatrix}. \]
Again, the automorphism \(\Lambda'(f)\) will be reduced exactly when each
of the monomials in \(\frac{\partial f}{\partial a^{*}}\) depends on at least
one of the variables \(b_{ij}^{*}\).

From the above results it is clear that the two groups \(\ST_{c}\) and
\(\oST_{c}\) are isomorphic. For the sequel it will be useful to define the
following explicit identification. Consider the isomorphism of linear spaces
\(\ol{\DR}^{0}(F_{r})\to \ol{\DR}^{0}(F_{r}^{*})\) obtained by mapping a
necklace word \(f(a,b_{ij})\) in \(\ol{\DR}^{0}(F_{r})\) to the necklace word
\(\tilde{f}\deq -f(a^{*},b_{ij}^{*})\) in \(\ol{\DR}^{0}(F_{r}^{*})\). We
denote by \(\app{o}{\ST_{c}}{\oST_{c}}\) the unique isomorphism of abelian
groups making the diagram
\begin{equation}
  \label{eq:11}
  \xymatrix{
    \ol{\DR}^{0}(F_{r}) \ar[d]_{\Lambda} \ar[rr]^{f\mapsto \tilde{f}} &&
    \ol{\DR}^{0}(F_{r}^{*}) \ar[d]^{\Lambda'}\\
    \ST_{c} \ar[rr]_{o} && \oST_{c}}
\end{equation}
commute, i.e. such that \(o(\Lambda(f)) = \Lambda'(\tilde{f})\).

\subsection{Affine symplectomorphisms}

Let \(\varphi\) be an affine automorphism of \(\bb{C}\Q_{r}\). Being
\(\bb{C}^{2}\)-linear, it must act as
\begin{equation}
  \label{eq:6}
  \begin{cases}
    a\mapsto A_{11}a + A_{12}a^{*} + B_{1}\\
    a^{*}\mapsto A_{21}a + A_{22}a^{*} + B_{2}
  \end{cases} \quad\text{ for some }
  \begin{pmatrix}
    A_{11} & A_{12}\\
    A_{21} & A_{22}
  \end{pmatrix}\in \GL_{2}(\bb{C}),\,
  \begin{pmatrix}
    B_{1}\\
    B_{2}
  \end{pmatrix}\in \bb{C}^{2}
\end{equation}
on the linear subspace spanned by \(a\) and \(a^{*}\) in \(\bb{C}\Q_{r}\),
whereas its action on the arrows \(1\to 2\) and \(2\to 1\) will be described
by the two \emph{complex} matrices \(M^{\varphi}\) and \(N^{\varphi}\in
\GL_{r}(\bb{C})\). Now, by theorem \ref{teo:aut-fix-cr2} the automorphism
\(\varphi\) fixes \(c_{r}^{2}\) if and only if \(M^{\varphi} =
(N^{\varphi})^{-1}\). Granted that, we can compute
\[ \varphi(c_{r}^{1}) = \varphi([a,a^{*}] - \sum_{\alpha=1}^{r} d_{\alpha}
b_{\alpha}) = \varphi([a,a^{*}]) - \sum_{\alpha,\beta,\gamma=1}^{r}
M^{\varphi}_{\alpha\beta}d_{\beta} b_{\gamma}N^{\varphi}_{\gamma\alpha}. \]
Since the entries of \(M^{\varphi}\) and \(N^{\varphi}\) are complex numbers
they commute with everything, so that
\[ \sum_{\alpha,\beta,\gamma=1}^{r} M^{\varphi}_{\alpha\beta}d_{\beta}
b_{\gamma}N^{\varphi}_{\gamma\alpha} = \sum_{\beta,\gamma} \sum_{\alpha}
N^{\varphi}_{\gamma\alpha} M^{\varphi}_{\alpha\beta} d_{\beta} b_{\gamma} = 
\sum_{\beta,\gamma} \delta_{\gamma\beta} d_{\beta} b_{\gamma} = \sum_{\beta}
d_{\beta} b_{\beta}. \]
Hence \(\varphi\) preserves \(c_{r}^{1}\) if and only if it preserves
\([a,a^{*}]\), and this happens if and only if the matrix \(A\) in the
transformation \eqref{eq:6} has determinant \(1\). Denoting by
\(\ASL_{2}(\bb{C})\) the subgroup of affine transformations of this form, we
conclude that the group of affine symplectic automorphisms of \(\bb{C}\Q_{r}\)
is the direct product
\[ \Aff_{c}\deq \ASL_{2}(\bb{C})\times \GL_{r}(\bb{C}). \]
\begin{lem}
  \label{lem:affc-red}
  An affine symplectic automorphism is reduced if and only if it fixes \(a\)
  and \(a^{*}\).
\end{lem}
Indeed every transformation of the form \eqref{eq:6} with \(\det A = 1\) comes
from an automorphism of \(\bb{C}\Q_{0}\). Thus the subgroup of \emph{reduced}
affine symplectic automorphisms coincides with the subgroup
\(\GL_{r}(\bb{C})\) of scalar invertible matrices inside
\(\GL_{r}(\mathcal{A}_{1})\).

\subsection{Tame symplectomorphisms}

Having defined some classes of ``nice'' symplectic automorphisms, we proceed
to consider the subgroup of \(\Aut(\bb{C}\Q_{r};c_{r})\) which is generated by
them.
\begin{defn}
  \label{def:tame-aut}
  A symplectic automorphism of \(\bb{C}\Q_{r}\) is called \textbf{tame} if it
  belongs to the subgroup generated by \(\ST_{c}\) and \(\Aff_{c}\).
\end{defn}
We denote by \(\TAut(\bb{C}\Q_{r};c_{r})\) the subgroup of tame symplectic
automorphisms. It is possible that \emph{every} symplectic automorphism of
\(\bb{C}\Q_{r}\) is tame (this is unknown even for \(r=2\)); anyway, this
issue will be irrelevant for what follows.

It is natural to ask if one can replace \(\ST_{c}\) with \(\oST_{c}\) in the
definition \ref{def:tame-aut}. When \(r\) is \emph{even} (in which case
\(\ceil{r/2} = \floor{r/2}\)), the answer is positive: let \(\mathcal{F}_{r}\)
denote the affine symplectomorphism of \(\bb{C}\Q_{r}\) defined by
\begin{equation}
  \label{eq:def-fft}
  \begin{cases}
    a\mapsto -a^{*}\\
    a^{*}\mapsto a
  \end{cases}
  \quad
  \begin{cases}
    x_{i}\mapsto -y_{i}^{*}\\
    x_{i}^{*}\mapsto y_{i}
  \end{cases}
  \quad
  \begin{cases}
    y_{i}\mapsto -x_{i}^{*}\\
    y_{i}^{*}\mapsto x_{i}
  \end{cases}
\end{equation}
We then have the following:
\begin{teo}
  \label{teo:gen-op-tr}
  The map \(\Aut(\bb{C}\Q_{r};c_{r})\to \Aut(\bb{C}\Q_{r};c_{r})\) defined by
  \(\psi\mapsto \mathcal{F}_{r}^{-1}\circ \psi\circ \mathcal{F}_{r}\)
  restricts to the isomorphism \(\app{o}{\ST_{c}}{\oST_{c}}\) defined by the
  diagram \eqref{eq:11}.
\end{teo}
\begin{proof}
  Let \(f\in \ol{\DR}^{0}(F_{r})\); we compute the action of \(\psi\deq
  \mathcal{F}_{r}^{-1}\circ \Lambda(f)\circ \mathcal{F}_{r}\) on the arrows of
  \(\Q_{r}\). It is immediate to check that \(\psi\) fixes the starred arrows.
  Now let \(p_{0}\deq \frac{\partial f}{\partial a}\), \(p_{ij}\deq
  \frac{\partial f}{\partial b_{ij}}\); they are non-commutative polynomials in
  the indeterminates \(a\) and \((x_{i}y_{j})_{i,j=1\dots r/2}\). Then, by
  direct computation,
  \[ \begin{aligned}
    \psi(a) &= \mathcal{F}_{r}^{-1}(-a^{*}-p_{0}(a,x_{1}y_{1},\dots,x_{r/2}y_{r/2}))
    = a - p_{0}(a^{*},y_{1}^{*}x_{1}^{*},\dots,y_{r/2}^{*}x_{r/2}^{*})\\
    \psi(x_{i}) &= \mathcal{F}_{r}^{-1}(-y_{i}^{*} - \sum_{k}
    p_{ki}(a,x_{1}y_{1},\dots,x_{r/2}y_{r/2}) x_{k}) = x_{i} - \sum_{k}
    p_{ki}(a^{*},y_{1}^{*}x_{1}^{*},\dots,y_{r/2}^{*}x_{r/2}^{*}) y_{k}^{*}\\
    \psi(y_{i}) &= \mathcal{F}_{r}^{-1}(-x_{i}^{*} - \sum_{k} y_{k}
    p_{ik}(a,x_{1}y_{1},\dots,x_{r/2}y_{r/2})) = y_{i} - \sum_{k} x_{k}^{*}
    p_{ik}(a^{*},y_{1}^{*}x_{1}^{*},\dots,y_{r/2}^{*}x_{r/2}^{*}).
  \end{aligned} \]
  To conclude it remains to observe that \(-p_{0}(a^{*}, y_{1}^{*}x_{1}^{*},
  \dots, y_{r/2}^{*}x_{r/2}^{*})\) coincides with the non-commutative
  polynomial \(\frac{\partial}{\partial a^{*}} \tilde{f}\), and similarly
  \(-p_{ij}(a^{*}, y_{1}^{*}x_{1}^{*}, \dots, y_{r/2}^{*}x_{r/2}^{*})\)
  coincides with \(\frac{\partial}{\partial b_{ij}^{*}} \tilde{f}\), so that
  \(\psi = \Lambda'(\tilde{f}) = o(\Lambda(f))\).
\end{proof}
It follows that every op-triangular symplectomorphism is tame and, vice versa,
every triangular symplectomorphism is generated by the affine and
op-triangular ones.

When \(r\) is \emph{odd} it is not clear if one can obtain the map \(o\) in a
similar way. As a matter of fact, in this case we do not know any general
recipe to express an op-triangular symplectomorphism as a composition of
affine and triangular ones. Fortunately, we will see in the next section that
when we restrict to the subgroup \(\mathcal{P}_{r}\) inside
\(\TAut(\bb{C}\Q_{r};c_{r})\) this difficulty disappears.

\section{The group $\mathcal{P}_{r}$}
\label{s:defP}

Let us denote by \(\STl_{c}\) the subgroup of \(\ST_{c}\) given by the image
under \(\Lambda\) of the subgroup of \(\ol{\DR}^{0}(F_{r})\) consisting of
necklace words of the form \(f = p(a)b_{11}\) for some \(p\in \bb{C}[a]\).
Clearly, \(\STl_{c}\) is isomorphic to the abelian group \((\bb{C}[a],+)\).
\begin{defn}
  \label{def:Pr}
  We denote by \(P_{r}\) the subgroup of \(\TAut(\bb{C}\Q_{r};c_{r})\)
  generated by \(\STl_{c}\) and the subgroup \(\Aff_{c}^{\red}\subseteq
  \Aff_{c}\) of reduced affine symplectic automorphisms.
\end{defn}
Since every automorphism in \(\STl_{c}\) is itself reduced (as every term in
\(\frac{\partial f}{\partial a}\) will contain an occurrence of \(b_{11}\)),
we observe that \(P_{r}\) is in fact a subgroup of \(K_{c_{r}}\). Our next aim
is to prove that this group is isomorphic to \(\GL_{r}(\bb{C}[a])\).
\begin{lem}
  The restriction of the action \eqref{eq:act-Aut-GL} to \(P_{r}\) is trivial
  on the subsets \(N(\STl_{c})\) and \(N(\Aff_{c}^{\red})\subseteq
  \GL_{r}(\mathcal{A}_{1})\).
\end{lem}
\begin{proof}
  It suffices to show that the generators of \(P_{r}\) act trivially. It is
  clear that any \(\bb{C}^{2}\)-linear automorphism of \(\bb{C}\Q_{r}\) acting
  entry-wise on a matrix in \(\GL_{r}(\mathcal{A}_{1})\) fixes every scalar
  matrix, hence fixes \(N(\Aff_{c}^{\red}) = \GL_{r}(\bb{C})\). Now let us
  take \(\psi\in \STl_{c}\); since \(\psi = \Lambda(p(a)b_{11})\) for some
  \(p\in \bb{C}[a]\), the matrix \(N^{\psi}\) will be of the form
  \begin{equation}
    \label{eq:12}
    N^{\psi} = \begin{pmatrix}
      1 & 0 & 0 & \hdots & 0\\
      p(a) & 1 & 0 & \hdots & 0\\
      0 & 0 & 1 & \hdots & 0\\
      \vdots & \vdots & \vdots & \ddots & \vdots\\
      0 & 0 & 0 & \hdots & 1
    \end{pmatrix}
  \end{equation}
  and hence only depends on \(a\). But then \(N^{\psi}\) is fixed both by
  elements of \(\Aff_{c}^{\red}\) (by lemma \ref{lem:affc-red}) and by
  elements of \(\STl_{c}\) (because every triangular automorphism fixes
  \(a\)).
\end{proof}
\begin{teo}
  \label{teo:Np-iso}
  The restriction of the map \(N\) to the subgroup \(P_{r}\) induces an
  isomorphism of groups \(P_{r}\to \GL_{r}(\bb{C}[a])\).
\end{teo}
\begin{proof}
  By a straightforward induction, the previous lemma implies that \(P_{r}\)
  acts trivially on the image \(N(P_{r})\subseteq \GL_{r}(\mathcal{A}_{1})\).
  This means that
  \[ N^{\psi_{2}\circ \psi_{1}} = N^{\psi_{2}} \psi_{2}(N^{\psi_{1}}) =
  N^{\psi_{2}} N^{\psi_{1}} \]
  for each \(\psi_{1},\psi_{2}\in P_{r}\), or in other words that the crossed
  morphism \(N\) becomes a genuine morphism of groups when restricted to
  \(P_{r}\). We claim that the kernel of this morphism is trivial. Indeed,
  suppose \(\psi\in P_{r}\) is such that \(N^{\psi} = I_{r}\); then by
  definition \(\psi\) fixes the arrows \(x_{i}\), \(x_{i}^{*}\), \(y_{j}\) and
  \(y_{j}^{*}\). Every automorphism in \(P_{r}\) fixes \(a\), hence \(\psi\)
  can only act nontrivially on \(a^{*}\), sending it to \(a^{*} + h\) for some
  \(h\in \mathcal{A}_{1}\). But \(\psi\) is symplectic, and in particular
  \(\psi(c_{r}^{1}) = c_{r}^{1}\), hence
  \[ [a,a^{*}] + [a,h] + \sum_{\alpha} d_{\alpha} b_{\alpha} = [a,a^{*}] +
  \sum_{\alpha} d_{\alpha}b_{\alpha} \]
  from which \([a,h] = 0\) follows. This implies that \(h\) is a polynomial in
  \(a\); then it must be zero, as \(\psi\) is reduced.
  
  It remains to prove that the image of \(\rist{N}{P_{r}}\) coincides with
  \(\GL_{r}(\bb{C}[a])\subseteq \GL_{r}(\mathcal{A}_{1})\). To do this, it is
  sufficient to show that it contains a set of generators for
  \(\GL_{r}(\bb{C}[a])\). Such a set is given (see e.g. \cite{hom89}) by the
  invertible diagonal matrices and by the (elementary) \emph{transvections},
  i.e. matrices of the form \(I_{r} + p\mathtt{e}_{\alpha\beta}\) where \(p\)
  is a polynomial and \(\mathtt{e}_{\alpha\beta}\) is the \(r\times r\) matrix
  with \(1\) at the position \((\alpha,\beta)\) and \(0\) elsewhere. Clearly,
  invertible diagonal matrices are contained in \(\GL_{r}(\bb{C}) =
  N(\Aff_{c}^{\red})\). Also, the matrix \eqref{eq:12} associated to
  \(\Lambda(p(a)b_{11})\in \STl_{c}\) is exactly the transvection by \(p\) in
  the \((2,1)\) plane of \(\bb{C}[a]^{r}\). But the symmetric group \(S_{r}\),
  embedded in \(\GL_{r}(\bb{C})\) as the subgroup of permutation matrices,
  acts transitively on the \(r\) axes of the \(\bb{C}[a]\)-linear space
  \(\bb{C}[a]^{r}\), so that by composing the transvection
  \(N^{\Lambda(p(a)b_{11})}\) with suitable permutation matrices in
  \(\GL_{r}(\bb{C})\) we obtain every possible transvection in
  \(\bb{C}[a]^{r}\), as we needed.
\end{proof}
Together with theorem \ref{teo:aut-fix-cr2}, this implies that that the
restriction of the crossed morphism \(M\) to \(P_{r}\) induces an isomorphism
\(P_{r}\to \GL_{r}(\bb{C}[a])^{\op}\).

Let us note another interesting consequence of theorem \ref{teo:Np-iso}.
Denote by
\[ \app{\Psi}{\GL_{r}(\bb{C}[a])}{P_{r}} \]
the inverse morphism to \(N\) on \(\GL_{r}(\bb{C}[a])\). By definition, given
a matrix \(A\in \GL_{r}(\bb{C}[a])\) the automorphism \(\Psi(A)\) will act on
the arrows \(1\to 2\) and \(2\to 1\) in \(\Q_{r}\) as \(B\mapsto BA\) and
\(D\mapsto A^{-1}D\) respectively, where \(B\) and \(D\) are the matrices
given by \eqref{eq:7}. Let us write again \(a^{*}\mapsto a^{*} + h\) for the
action of \(\Psi(A)\) on \(a^{*}\). As \(\Psi(A)\) fixes \(c_{r}^{1}\), the
following equality in \(\mathcal{A}_{1}\) must hold:
\begin{equation}
  \label{eq:cond-psi}
  [a,h] = \sum_{\alpha,\beta,\gamma} (A^{-1}_{\alpha\beta} e_{\beta\gamma}
  A_{\gamma\alpha} + e_{\alpha\alpha}).
\end{equation}
Then theorem \ref{teo:Np-iso} can be rephrased by saying that, for every
\(A\in \GL_{r}(\bb{C}[a])\), equation \eqref{eq:cond-psi} can be solved
uniquely for \(h\). This appears to be quite non-trivial to prove directly.
\begin{defn}
  We denote by \(\mathcal{P}_{r}\) the subgroup of
  \(\Aut(\bb{C}\Q_{r};c_{r})\) obtained by replacing, in the semidirect
  product \eqref{eq:sdp-simp}, the group \(K_{c_{r}}\) with its subgroup
  \(P_{r}\).
\end{defn}
It follows that every element in \(\mathcal{P}_{r}\) can be written in a
unique way as the composition \(\psi_{0}\circ \psi\), with \(\psi_{0}\in
\Aut(\bb{C}\Q_{0};c_{0})\) and \(\psi\in P_{r}\).
\begin{lem}
  \label{lem:aut0-tame}
  Every automorphism of \(\Aut(\bb{C}\Q_{0};c_{0})\) is tame.
\end{lem}
\begin{proof}
  It follows from results of Makar-Limanov \cite{ml70,ml84} that the group
  \(\Aut(\bb{C}\Q_{0};c_{0})\) is isomorphic to the group of automorphisms of
  the first Weyl algebra. Moreover, in \cite{dix68} Dixmier proved that the
  latter group is generated by the family of triangular automorphisms
  \((a,a^{*})\mapsto (a, a^{*} + \frac{\partial}{\partial a}p(a))\) indexed by
  a polynomial \(p\in a\bb{C}[a]\) and by the single automorphism
  \(\mathcal{F}_{0}\) defined by \((a,a^{*})\mapsto (-a^{*},a)\). It is
  immediate to check that all these automorphisms, when extended from
  \(\bb{C}\Q_{0}\) to \(\bb{C}\Q_{r}\), are tame.
\end{proof}
From this lemma it follows that \(\mathcal{P}_{r}\) is a subgroup of
\(\TAut(\bb{C}\Q_{r};c_{r})\). Notice that the automorphism
\(\mathcal{F}_{0}\) considered in the previous proof is just the affine
symplectomorphism \(\mathcal{F}_{r}\) defined by \eqref{eq:def-fft} modulo the
ideal \(\mathcal{I}_{r}\).

We conclude this section by clarifying the r\^ole of op-triangular
symplectomorphisms inside \(\mathcal{P}_{r}\). By analogy with \(\STl_{c}\),
let us denote by \(\oSTl_{c}\) the subgroup of \(\oST_{c}\) generated by the
image under \(\Lambda'\) of necklace words of the form \(p(a^{*})b_{11}^{*}\).
As
\[ N^{\Lambda'(p(a^{*})b_{11}^{*})} = 
\begin{pmatrix}
  1 & p(a^{*}) & 0 & \hdots & 0\\
  0 & 1 & 0 & \hdots & 0\\
  0 & 0 & 1 & \hdots & 0\\
  \vdots & \vdots & \vdots & \ddots & \vdots\\
  0 & 0 & 0 & \hdots & 1
\end{pmatrix} \]
it is clear that \(\oSTl_{c}\) is isomorphic to the abelian group
\((\bb{C}[a^{*}],+)\).

The reader can easily convince himself or herself that all the steps leading
to theorem \ref{teo:Np-iso} can be carried out equally well for the subgroup
\(P_{r}'\) of \(\TAut(\bb{C}\Q_{r};c_{r})\) generated by \(\oSTl_{c}\) and
\(\Aff_{c}^{\red}\), replacing the group \(\GL_{r}(\bb{C}[a])\) with
\(\GL_{r}(\bb{C}[a^{*}])\subseteq \GL_{r}(\mathcal{A}_{1})\). In particular,
the restriction of \(N\) to \(P_{r}'\) induces an isomorphism \(P_{r}'\to
\GL_{r}(\bb{C}[a^{*}])\).

Suppose now \(r\geq 2\). Let \(\varphi\) stand for the affine
symplectomorphism acting as \(\mathcal{F}_{2}\) on the arrows \(a\),
\(x_{1}\), \(y_{1}\) and their starred version:
\[ \begin{cases}
  a\mapsto -a^{*}\\
  a^{*}\mapsto a
\end{cases}
\quad
\begin{cases}
  x_{1}\mapsto -y_{1}^{*}\\
  x_{1}^{*}\mapsto y_{1}
\end{cases}
\quad
\begin{cases}
  y_{1}\mapsto -x_{1}^{*}\\
  y_{1}^{*}\mapsto x_{1}
\end{cases} \]
and fixing all the other arrows in \(\bb{C}\Q_{r}\). Then the same
calculations used in the proof of theorem \ref{teo:gen-op-tr} show that,
independently from the parity of \(r\), the following equality holds:
\begin{equation}
  \label{eq:15}
  \varphi^{-1}\circ \Lambda(p(a)b_{11}) \circ \varphi =
  \Lambda'(-p(a^{*})b_{11}^{*}).
\end{equation}
As \(\varphi\in \mathcal{P}_{r}\) we conclude that every generator of
\(\oSTl_{c}\) belongs to \(\mathcal{P}_{r}\), hence \(P_{r}'\subseteq
\mathcal{P}_{r}\). Of course, equation \eqref{eq:15} also shows that one can
equivalently define the group \(\mathcal{P}_{r}\) as the semidirect product
\(P_{r}'\rtimes \Aut(\bb{C}\Q_{0};c_{0})\). In other words, once we have the
semidirect factor \(\Aut(\bb{C}\Q_{0};c_{0})\) at our disposal, the two groups
\(P_{r}\) and \(P_{r}'\) are completely interchangeable.

\section{The action of $\mathcal{P}_{r}$ on Gibbons-Hermsen manifolds}
\label{s:action}

In this section we are concerned with the action of (tame) symplectic
automorphisms of \(\bb{C}\Q_{r} = \bb{C}\ol{Z}_{r}\) on the manifolds
\(\mathcal{C}_{n,r}\). Let us start by recalling how this action is defined.

To begin with, we should make explicit the embedding of \(\mathcal{C}_{n,r}\)
inside the moduli space of representations of the quiver \(\ol{Z}_{r}\). To do
this we need a bijection between \(\Rep(\ol{Z}_{r},(n,1))\) and the linear
space \(V_{n,r}\) defined by \eqref{eq:def-Vnr}. Let us denote a point in
\(\Rep(\ol{Z}_{r},(n,1))\) by a \(2(r+1)\)-tuple of the form
\[ (A, \bar{A}, X_{1}, \dots, X_{\ceil{r/2}}, \bar{X}_{1}, \dots,
\bar{X}_{\ceil{r/2}}, Y_{1}, \dots, Y_{\floor{r/2}}, \bar{Y}_{1}, \dots,
\bar{Y}_{\floor{r/2}}) \]
where the arrow \(a\) is represented by the matrix \(A\), the arrow \(a^{*}\)
by the matrix \(\bar{A}\) and so on. We identify this point with a quadruple
\((X,Y,v,w)\in V_{n,r}\) using the following correspondence:
\begin{equation}
  \label{eq:m1}
  \begin{cases}
    A\leftrightarrow X\\
    \bar{A}\leftrightarrow Y
  \end{cases}
  \quad
  \begin{cases}
    X_{i}\leftrightarrow -v_{\bullet,2i-1}\\
    \bar{X}_{i}\leftrightarrow w_{2i-1,\bullet}
  \end{cases}
  \quad
  \begin{cases}
    Y_{j}\leftrightarrow w_{2j,\bullet}\\
    \bar{Y}_{j}\leftrightarrow v_{\bullet,2j}
  \end{cases}
  \quad
  \begin{array}{l}
    i=1\dots \ceil{r/2}\\
    j=1\dots \floor{r/2}.
  \end{array}
\end{equation}
In other words, we build the \(n\times r\) matrix \(v\) using the \(r\)
columns \(-X_{1}\), \(\bar{Y}_{1}\), \(-X_{2}\), \(\bar{Y}_{2}\) and so on;
similarly, we build the \(r\times n\) matrix \(w\) using the \(r\) rows
\(\bar{X}_{1}\), \(Y_{1}\), \(\bar{X}_{2}\), \(Y_{2}\) and so on.

We can sum up the correspondence between the various arrows in \(\Q_{r}\)
and their representative matrices as follows:
\begin{center}
  \begin{tabular}{l|cccccc}
    arrow in \(\ol{Z}_{r} = \Q_{r}\) & \(a\) & \(a^{*}\) &
    \(x_{i} = -d_{2i-1}\) & \(x_{i}^{*} = b_{2i-1}\) &
    \(y_{j} = b_{2j}\) & \(y_{j}^{*} = d_{2j}\)\\
    matrix in \(\Rep(\ol{Z}_{r},(n,1))\) & \(A\) & \(\bar{A}\) &
    \(X_{i}\) & \(\bar{X}_{i}\) & 
    \(Y_{j}\) & \(\bar{Y}_{j}\)\\
    matrix in \(\mathcal{C}_{n,r}\) & \(X\) & \(Y\) &
    \(-v_{\bullet 2i-1}\) & \(w_{2i-1 \bullet}\) &
    \(w_{2j \bullet}\) & \(v_{\bullet 2j}\)
  \end{tabular}
\end{center}
Notice that for every \(\alpha=1\dots r\) the arrow \(d_{\alpha}\) is
represented on \(V_{n,r}\) by the column matrix \(v_{\bullet\alpha}\), and the
arrow \(b_{\alpha}\) by the row matrix \(w_{\alpha\bullet}\).

As explained in the introduction, on \(\Rep(\ol{Z}_{r},(n,1))\) there is a
natural action of the group \(\GL_{(n,1)}(\bb{C})\isom \GL_{n}(\bb{C})\); the
corresponding moment map is
\[ J(A,\bar{A},X_{i},\bar{X}_{i},Y_{j},\bar{Y}_{j}) = [A,\bar{A}] + \sum_{i}
X_{i}\bar{X}_{i} - \sum_{j} \bar{Y}_{j}Y_{j}. \]
Under the correspondence \eqref{eq:m1}, this becomes
\[ J(X,Y,v,w) = [X,Y] - \sum_{\alpha=1}^{r} v_{\bullet\alpha} w_{\alpha\bullet}, \]
which is exactly the map \eqref{eq:mom-mu}. We conclude that the symplectic
quotient
\[ J^{-1}(\tau I_{n})/\GL_{n}(\bb{C}) \]
inside \(\Rep(\ol{Z}_{r},(n,1))\) is isomorphic to the manifold
\(\mathcal{C}_{n,r}\), as defined by the quotient \eqref{eq:Cnr}.

We can now explain how the group of symplectomorphisms of the path algebra
\(\bb{C}\Q_{r}\) acts on \(\mathcal{C}_{n,r}\) for each \(n\geq 1\). Given
\(\psi\in \Aut(\bb{C}\Q_{r};c_{r})\), we consider for each arrow \(\xi\) in
\(\Q_{r}\) the non-commutative polynomial \(\psi(\xi)\). Given a point \(p =
(A, \bar{A}, X_{i}, \bar{X}_{i}, Y_{j}, \bar{Y}_{j})\in \Rep(\Q_{r},(n,1))\),
we can evaluate the polynomial \(\psi(\xi)\) at \(p\) (by substituting each
matrix in \(p\) for the arrow it represents) to obtain another matrix
\(\psi(\xi)(p)\). We define our action by declaring that \(\psi\) sends the
point \(p\) to the point
\[ (\psi(a)(p), \psi(a^{*})(p), \psi(x_{i})(p), \psi(x_{i}^{*})(p),
\psi(y_{j})(p), \psi(y_{j}^{*})(p)) \]
in \(\Rep(\Q_{r},(n,1))\). One can easily verify that, when \(\psi\) is
symplectic, this action restricts to an action on each fiber of the moment map
\(J\) which is constant along the orbits of \(\GL_{n}(\bb{C})\). The action on
\(\mathcal{C}_{n,r}\) is just the induced action on the quotient. As this is
naturally a \emph{right} action, we will write the image of the point \(p\)
under the action of \(\psi\) as \(p.\psi\).
\begin{ex}
  Let \(\psi\) denote the triangular automorphism of \(\bb{C}\Q_{3}\) given by
  \(\Lambda(a^{2}b_{21})\). As
  \[ \psi(a,a^{*},x_{1},x_{2},x_{1}^{*},x_{2}^{*},y,y^{*}) = (a, a^{*} +
  ax_{2}y + x_{2}ya, x_{1}, x_{2}, x_{1}^{*}, x_{2}^{*} + ya^{2}, y, y^{*} +
  a^{2}x_{2}) \]
  we have that
  \[ (A,\bar{A},X_{1},X_{2},\bar{X}_{1},\bar{X}_{2},Y,\bar{Y}).\psi = (A,
  \bar{A} + AX_{2}Y + X_{2}YA, X_{1}, X_{2}, \bar{X}_{1}, \bar{X}_{2} +
  YA^{2}, Y, \bar{Y} + A^{2}X_{2}) \]
  or equivalently, in terms of the coordinates on \(V_{n,r}\),
  \[ (X,Y,v,w).\psi = (X, Y - X v\mathtt{e}_{32}w - v\mathtt{e}_{32}w X, 
  v - X^{2}v\mathtt{e}_{32}, w + \mathtt{e}_{32}wX^{2}) \]
  where again we are denoting by \(\mathtt{e}_{\alpha\beta}\) the matrix with
  \(1\) at the position \((\alpha,\beta)\) and \(0\) elsewhere.
\end{ex}
It is not difficult to guess how a reduced affine symplectomorphism acts on
\(\mathcal{C}_{n,r}\). In what follows, we will denote by \(\trasp{m}\) the
transpose of a matrix \(m\).
\begin{lem}
  \label{lem:act-aff}
  Let \(\varphi\in \Aff_{c}^{\red}\) and \(T\) be the corresponding matrix in
  \(\GL_{r}(\bb{C})\). Then \((X,Y,v,w).\varphi = (X, Y, v \trasp{T}^{-1},
  \trasp{T} w)\).
\end{lem}
\begin{proof}
  By definition, \(T\deq N^{\varphi}\) is the only matrix such that
  \begin{equation}
    \label{eq:17}
    \varphi(b_{\alpha}) = b_{1} T_{1\alpha} + \dots + b_{r} T_{r\alpha}.
  \end{equation}
  As the arrow \(b_{\alpha}\) is represented on \(V_{n,r}\) by the row matrix
  \(w_{\alpha\bullet}\), a simple computation shows that the row matrix
  representing \(\varphi(b_{\alpha})\), as expressed by \eqref{eq:17},
  coincides with the \(\alpha\)-th row of the matrix \(\trasp{T}w\).
  Similarly, \(T^{-1} = M^{\varphi}\) is the only matrix such that
  \begin{equation}
    \label{eq:18}
    \varphi(d_{\alpha}) = (T^{-1})_{\alpha 1} d_{1} + \dots + (T^{-1})_{\alpha
      r} d_{r}.
  \end{equation}
  As the arrow \(d_{\alpha}\) is represented on \(V_{n,r}\) by the column
  matrix \(v_{\bullet\alpha}\), it is again easy to verify that the column
  matrix representing \(\varphi(d_{\alpha})\), as expressed by \eqref{eq:18},
  coincides with the \(\alpha\)-th column of the matrix \(v\trasp{T}^{-1}\).
\end{proof}
It would be tempting to conjecture, by analogy with the result proved in
\cite{bp08}, that the group \(\TAut(\bb{C}\Q_{r};c_{r})\) acts transitively on
(the disjoint union over \(n\) of) the manifolds \(\mathcal{C}_{n,r}\) for
every \(r\). However, it does not seem feasible to generalize the proof given
in \cite{bp08} for \(r=2\) to the new setting. Instead, following \cite{mt13},
we will study this problem on the ``large'' open subset
\(\mathcal{R}_{n,r}\subset \mathcal{C}_{n,r}\) which was defined in the
introduction (cf. \eqref{eq:2}) as the set of points \([X,Y,v,w]\in
\mathcal{C}_{n,r}\) such that at least one of the matrices \(X\) and \(Y\) is
\emph{regular semisimple}, i.e. diagonalizable with distinct eigenvalues. This
has the advantage of allowing explicit computations. Our result is that, as in
the case \(r=2\), to connect each pair of points in \(\mathcal{R}_{n,r}\) it
suffices to take the (much smaller) subgroup \(\mathcal{P}_{r}\) inside
\(\TAut(\bb{C}\Q_{r};c_{r})\).
\begin{teo}
  \label{teo:tr-act}
  For every pair of points \(p,p'\in \mathcal{R}_{n,r}\) there exists
  \(\psi\in \mathcal{P}_{r}\) such that \(p.\psi = p'\).
\end{teo}
\begin{proof}
  The basic strategy is the same as the one used in \cite[Lemma 8.4]{bp08} and
  \cite{mt13}, i.e. a reduction to the case \(r=1\). Inside
  \(\mathcal{R}_{n,r}\) there is the submanifold
  \[ \mathcal{N}_{n,r}\deq \set{[X,Y,v,w]\in \mathcal{R}_{n,r} | v_{\bullet r}
    = 0 \text{ and } w_{r\bullet} = 0} \]
  which is isomorphic to \(\mathcal{R}_{n,r-1}\). Now suppose that, for every
  \(r>1\), we are able to move every point \(p\in \mathcal{R}_{n,r}\) to
  \(\mathcal{N}_{n,r}\) using a symplectomorphism in \(\mathcal{P}_{r}\).
  Notice that \(\mathcal{P}_{r-1}\) naturally embeds into \(\mathcal{P}_{r}\)
  as the subgroup of automorphisms fixing the two arrows in \(\bb{C}\Q_{r}\)
  that do not appear in \(\bb{C}\Q_{r-1}\). Then a straightforward induction
  on \(r\) shows that every point in \(\mathcal{R}_{n,r}\) can be moved inside
  (an isomorphic copy of) \(\mathcal{R}_{n,1}\). As the subgroup
  \(\Aut(\bb{C}\Q_{0};c_{0})\subseteq \mathcal{P}_{r}\) acts transitively on
  \(\mathcal{C}_{n,1}\supseteq \mathcal{R}_{n,1}\) \cite{bw00}, the theorem
  follows.

  So let us take \(r>1\), \(p\in \mathcal{R}_{n,r}\) and let \((X,Y,v,w)\) be
  a representative of the point \(p\). We can assume \(p\) to be in
  \(\mathcal{C}_{n,r}'\) (otherwise we only need to act with
  \(\mathcal{F}_{0}\), which exchanges \(\mathcal{C}_{n,r}'\) and
  \(\mathcal{C}_{n,r}''\)) so that \(X = \diag(\lambda_{1}, \dots,
  \lambda_{n})\) with \(\lambda_{i}\neq \lambda_{j}\) for \(i\neq j\). Also,
  recall from the introduction that for a point \(p\in \mathcal{C}_{n,r}'\) we
  have that \(v_{k\bullet}w_{\bullet k} = -\tau\) for every \(k=1\dots n\). As
  \(\tau\neq 0\), this implies that each row of the matrix \(v\) and each
  column of the matrix \(w\) have at least a nonzero entry. Clearly, this
  result also holds when \(p\in \mathcal{C}_{n,r}''\).

  Let us suppose that \(r\) is odd, say \(r = 2s + 1\). Then the last column
  of \(v\) represents the arrow \(x_{s+1}\) and the last row of \(w\)
  represents the arrow \(x_{s+1}^{*}\). As each column of \(w\) has a nonzero
  entry, there exists a matrix \(T\in \GL_{r}(\bb{C})\) such that each entry
  in the \(2s\)-th row of \(Tw\) is nonzero. The matrix \(T\) corresponds, via
  lemma \ref{lem:act-aff}, to some reduced affine symplectomorphism
  \(\varphi\); acting with it, we arrive at a point \((X,Y,v,w')\) such that
  \(X = \diag(\lambda_{1}, \dots, \lambda_{n})\) and \(w'_{2s,k}\neq 0\) for
  every \(k=1\dots n\). We claim that there exists a unique polynomial \(p\)
  of degree \(n-1\) such that
  \begin{equation}
    \label{eq:19}
    w_{2s+1,\bullet} + w_{2s,\bullet} p(X) = 0.
  \end{equation}
  Indeed, as \(X\) is diagonal and the row vector \(w_{2s,\bullet}\) has
  nonzero entries, to solve equation \eqref{eq:19} means to find a polynomial
  whose value at \(\lambda_{k}\) is given by \(-w_{2s+1,k}/w_{2s,k}\) for
  every \(k=1\dots n\), and it is well known that such an interpolation
  problem always has a unique solution of the stated degree. Then, acting with
  the triangular symplectomorphism
  \[ \Lambda(p(a)b_{s+1,s}) = 
  \begin{cases}
    a^{*}\mapsto a^{*} + \frac{\partial }{\partial a} (p(a)x_{s+1}y_{s})\\
    x_{s+1}^{*}\mapsto x_{s+1}^{*} + y_{s}p(a)\\
    y_{s}^{*}\mapsto y_{s}^{*} + p(a)x_{s+1}
  \end{cases} \]
  we get to a point \((X,Y',v',w'')\) such that \(w_{2s+1,\bullet} = 0\). Now
  we exchange \(X\) and \(Y'\) using \(\mathcal{F}_{0}\) and we repeat the
  same algorithm to kill the last row of \(v'\). Namely, as each row of \(v\)
  has a nonzero entry there exists a matrix \(T\in \GL_{r}(\bb{C})\) such that
  every entry in the \(2s\)-th column of \(vT\) is nonzero. Let us act with
  the affine symplectic automorphism corresponding to \(T\); then we are at a
  point \((Y',-X,v'',w'')\) such that \(v''_{k,2s}\neq 0\) for every
  \(k=1\dots n\). Again, this implies that there exists a unique polynomial
  \(q\) of degree \(n-1\) such that
  \[ -v_{\bullet,2s+1} + q(Y) v_{\bullet,2s} = 0, \]
  and acting with the op-triangular symplectomorphisms
  \[ \Lambda'(q(a^{*})b_{s+1,s}^{*}) =
  \begin{cases}
    a\mapsto a + \frac{\partial }{\partial
      a^{*}}(q(a^{*})y_{s}^{*}x_{s+1}^{*})\\
    x_{s+1}\mapsto x_{s+1} + q(a^{*}) y_{s}^{*}\\
    y_{s}\mapsto y_{s} + x_{s+1}^{*} q(a^{*})
  \end{cases} \]
  (which fixes \(x_{s+1}^{*}\), hence the last row of \(w\)) we finally arrive
  at a point in \(\mathcal{N}_{n,r}\), as we wanted.

  When \(r\) is even, say \(r = 2s\), the proof proceeds along completely
  analogous lines. In this case the last column of \(v\) represents the arrow
  \(y^{*}_{s}\) and the last row of \(w\) represents the arrow \(y_{s}\). The
  reduction from \(\mathcal{R}_{n,r}\) to \(\mathcal{N}_{n,r}\) is again
  accomplished in two steps. In the first we make every entry of
  \(v_{\bullet,2s}\) nonzero by acting with a suitable element of
  \(\Aff_{c}^{\red}\), and then act with
  \[ \Lambda(p(a)b_{ss}) =
  \begin{cases}
    a^{*}\mapsto a^{*} + \frac{\partial }{\partial a}(p(a)x_{s}y_{s})\\
    x_{s}^{*}\mapsto x_{s}^{*} + y_{s}p(a)\\
    y_{s}^{*}\mapsto y_{s}^{*} + p(a)x_{s}
  \end{cases} \]
  where the polynomial \(p\) of degree \(n-1\) is determined by solving the
  interpolation problem posed by the system
  \[ v_{\bullet,2s} - p(X)v_{\bullet,2s-1} = 0. \]
  With this step we reach a point \((X,Y',v',w')\) such that
  \(v'_{\bullet,2s}=0\). In the second step we use \(\mathcal{F}_{0}\) to get
  into \(\mathcal{C}_{n,r}''\), make all the entries of \(w_{2s-1,\bullet}\)
  different from zero in the usual manner, determine the polynomial \(q\) of
  degree \(n-1\) such that
  \[ w_{2s,\bullet} + w_{2s-1,\bullet}q(Y) = 0, \]
  and act with
  \[ \Lambda'(q(a^{*})b_{ss}^{*}) =
  \begin{cases}
    a\mapsto a + \frac{\partial }{\partial
      a^{*}}(q(a^{*})y_{s}^{*}x_{s}^{*})\\
    x_{s}\mapsto x_{s} + q(a^{*}) y_{s}^{*}\\
    y_{s}\mapsto y_{s} + x_{s}^{*} q(a^{*})
  \end{cases} \]
  to finally land into \(\mathcal{N}_{n,r}\).
\end{proof}
\begin{rem}
  As in \cite{mt13}, we should emphasize that the subset \(\mathcal{R}_{n,r}\)
  is \emph{not} closed under the action of \(\mathcal{P}_{r}\). This is clear
  from the fact that the subgroup \(\Aut(\bb{C}\Q_{0};c_{0})\) can send a
  point in \(\mathcal{R}_{n,1}\) to a point such that both \(X\) and \(Y\) are
  nilpotent (cf. \cite{bw00}).
\end{rem}
Before concluding this section, let us briefly explain the relation between
the group \(\mathcal{P}_{r}\) and the flows determined by the Hamiltonians
\eqref{eq:20} of the Gibbons-Hermsen hierarchy. First, notice that
\[ J_{k,I} = \tr Y^{k}vw = \tr Y^{k}([X,Y] - \tau I) = -\tau \tr Y^{k}, \]
so that the flow of a Hamiltonian of this form coincides with the action of
the op-triangular symplectomorphism \(\Lambda'(-\tau a^{*k})\).

Consider now the functions \(J_{k,m}\) for \(m = \mathtt{e}_{\alpha\beta}\).
When \(\alpha=\beta\) the corresponding flow is not polynomial, hence it
cannot be realized as the action of an element of \(\mathcal{P}_{r}\). When
\(\alpha\neq\beta\), the flow of \(J_{k,\mathtt{e}_{\alpha\beta}}\) is given
by
\[ \begin{aligned}
  X(t) &= X + t \sum_{i=1}^{k} Y^{k-i} v_{\bullet\alpha} w_{\beta\bullet} Y^{i-1}\\
  v_{\bullet\beta}(t) &= v_{\bullet \beta} - t Y^{k} v_{\bullet \alpha}\\
  w_{\alpha\bullet}(t) &= w_{\alpha\bullet} + t w_{\beta\bullet} Y^{k}.
\end{aligned} \]
In particular, when \(\alpha\) is even and \(\beta\) is odd it coincides with
the action of the op-triangular automorphism
\[ \begin{cases}
  a\mapsto a + t \frac{\partial }{\partial a^{*}} (a^{*k} y_{\alpha/2}^{*}
  x_{(\beta+1)/2}^{*})\\
  x_{(\beta+1)/2}\mapsto x_{(\beta+1)/2} + t a^{*k} y_{\alpha/2}^{*}\\
  y_{\alpha/2}\mapsto y_{\alpha/2} + t x_{(\beta+1)/2}^{*} a^{*k}
\end{cases} = \Lambda'(t a^{*k} b_{\frac{\beta+1}{2},\frac{\alpha}{2}}^{*}). \]
For different parities of \(\alpha\) and \(\beta\) the corresponding
automorphism is no longer op-triangular, but it still belongs to
\(\mathcal{P}_{r}\). Indeed, as \(\alpha\neq \beta\) we can always find a
permutation matrix \(P\) such that \(\mathtt{e}_{\alpha\beta} =
P^{-1}\mathtt{e}_{21}P\), so that
\[ J_{k,\mathtt{e}_{\alpha\beta}} = \tr Y^{k}vP^{-1}\mathtt{e}_{21}Pw = \tr
Y^{k}\tilde{v}\mathtt{e}_{21}\tilde{w} \]
with \(\tilde{v} = vP^{-1}\) and \(\tilde{w} = Pw\). By lemma
\ref{lem:act-aff}, the transformation \((v,w)\mapsto (\tilde{v},\tilde{w})\)
corresponds to the action of the unique reduced affine symplectomorphism
\(\varphi\) such that \(N^{\varphi} = \trasp{P}\).

Completely analogous results hold for the flow of Hamiltonian functions of the
form \(\tr X^{k}\) or \(\tr X^{k} v\mathtt{e}_{\alpha\beta}w\) (again with
\(\alpha\neq \beta\)). In particular when \(\alpha\) is odd and \(\beta\) is
even such a flow will coincide with the action of the triangular automorphism
\(\Lambda(t a^{k} b_{\frac{\alpha+1}{2},\frac{\beta}{2}})\).

\section{Some examples}
\label{s:ex}

Let us illustrate some of the results obtained above for the first few values
of \(r\). The case \(r=1\) is somewhat degenerate, but we will treat it for
completeness. The first zigzag quiver is
\[ Z_{1} = \xymatrix{
  1\ar@(ul,dl)[]_{a} & 2 \ar[l]_{x}}. \]
As \(q_{1} = 1\cdot 0 = 0\), triangular symplectomorphisms of \(\bb{C}\Q_{1}\)
are labeled by an element of \(\ol{\DR}^{0}(\bb{C}[a])\), i.e. a polynomial
without constant term. The only reduced symplectomorphisms are the affine
ones, acting as \((x,x^{*})\mapsto (\lambda x,\lambda^{-1}x^{*})\) for some
\(\lambda\in \bb{C}^{*}\); this action is trivial on representation spaces, so
that the group acting on \(\mathcal{C}_{n,1}\) effectively reduces to
\(\Aut(\bb{C}\Q_{0};c_{0})\) and the action is the same as the one first
defined by Berest and Wilson in \cite{bw00}. All this is well known, and
follows from the fact (mentioned in the introduction) that the rank 1
Gibbons-Hermsen system coincides with the rational Calogero-Moser system. The
action of the op-triangular automorphism
\[ (a,a^{*},x)\mapsto (a + p'(a^{*}),a^{*},x) \]
determined by a polynomial \(p\in \ol{\DR}^{0}(\bb{C}[a^{*}])\) corresponds to
a finite linear combination of Calogero-Moser flows.

Let us proceed to the more interesting case \(r=2\). The quiver
\[ Z_{2} = \xymatrix{
  1 \ar@(ul,dl)[]_{a} \ar@/_/[r]_{y} & 2 \ar@/_/[l]_{x}} \]
coincides with the one studied in \cite{bp08} and \cite{mt13}. As \(q_{2} =
1\), triangular symplectomorphisms of \(\bb{C}\Q_{2}\) are labeled by a
necklace word in the two generators \(a\) and \(b\deq b_{11} = xy\). One can
easily verify that the isomorphism \(P_{2}\isom \GL_{2}(\bb{C}[a])\) described
in section \ref{s:defP} identifies the lower unitriangular matrix \(\left(
  \begin{smallmatrix}
    1 & 0\\
    p(a) & 1
  \end{smallmatrix}
\right)\) with the triangular automorphism \(\Lambda(p(a)b)\) acting as
\[ (a^{*},x^{*},y^{*})\mapsto (a^{*} + \frac{\partial }{\partial a} (p(a)xy),
x^{*} + yp(a), y^{*} + p(a)x), \]
and the upper unitriangular matrix \(\left(
  \begin{smallmatrix}
    1 & p(a)\\
    0 & 1
  \end{smallmatrix}
\right)\) with the automorphism
\[ (a^{*},x,y)\mapsto (a^{*} + \frac{\partial}{\partial a} (p(a)xy), x +
p(a)y^{*}, y + x^{*}p(a)). \]
Similarly, the isomorphism \(P'_{2}\isom \GL_{2}(\bb{C}[a^{*}])\) identifies
the upper unitriangular matrix \(\left(
  \begin{smallmatrix}
    1 & p(a^{*})\\
    0 & 1
  \end{smallmatrix}
\right)\) with the op-triangular automorphism \(\Lambda'(p(a^{*})b^{*})\), and
the lower unitriangular matrix \(\left(
  \begin{smallmatrix}
    1 & 0\\
    p(a^{*}) & 1
  \end{smallmatrix}
\right)\) with the automorphism
\[ (a,x^{*},y^{*})\mapsto (a + \frac{\partial}{\partial a^{*}}
(p(a^{*})y^{*}x^{*}), x^{*} + yp(a^{*}), y^{*} + p(a^{*})x). \]
The map \(\app{k}{\GL_{2}(\bb{C}[a^{*}])}{\TAut(\bb{C}\Q_{2};c_{2})}\) defined
in the main proof of \cite{mt13} is essentially the inverse of
\(\rist{M}{P_{2}'}\); in particular it is injective (as it was conjectured
there). The qualifier ``essentially'' in the previous sentence is due to the
unfortunate choice we made of inverting the map \(M\) (which is naturally an
\emph{anti}-morphism of groups, at least using the standard ordering for the
composition of maps) instead of the map \(N\). As a consequence, the matrices
used in \cite{mt13} are actually the \emph{transposes} of the ones used here.
The notation adopted in the present paper is (hopefully) more consistent.

It is also easy to check that the group \(\mathcal{P}\) defined in \cite{mt13}
coincides (apart from the quotienting out of the subgroup of scalar affine
symplectomorphisms) with \(\mathcal{P}_{r}\) for \(r=2\). Indeed, in
\cite{mt13} we denoted by \(\mathcal{P}\) the subgroup of
\(\TAut(\bb{C}\Q_{2};c_{2})\) generated by the reduced affine
symplectomorphisms, the op-triangular symplectomorphisms of the form
\(\Lambda'(p(a^{*}))\) and \(\Lambda'(p(a^{*})b^{*})\) for some polynomial
\(p\), and the single affine symplectomorphism \(\mathcal{F}_{2}\) defined
as in \eqref{eq:def-fft}. All these automorphisms belong to
\(\mathcal{P}_{2}\), as it is immediate to verify. Vice versa, every element
of \(\mathcal{P}_{2}\) can be written as \(\psi_{0}\circ \psi\) with
\(\psi_{0}\in \Aut(\bb{C}\Q_{0};c_{0})\) and \(\psi\in P_{2}'\); both of these
groups are contained in \(\mathcal{P}\), as defined above.
\begin{rem}
  In \cite{mt13} we were actually interested in embedding into
  \(\TAut(\bb{C}\Q_{2};c_{2})\) the larger group \(\Gamma^{\alg}(2)\), where
  \[ \Gamma^{\alg}(r)\deq \{e^{p}I_{r}\}_{p\in z\bb{C}[z]} \times
  \PGL_{r}(\bb{C}[z]) \]
  is the group introduced by Wilson in \cite{wils09} (see \cite{mt13} for its
  precise definition). From the results proved in Section \ref{s:defP} it
  follows that also for every \(r>2\) the group \(\mathcal{P}_{r}\) contains
  (roughly speaking) ``two copies'' of \(\Gamma^{\alg}(r)\), one for \(z=a\)
  and the other for \(z=a^{*}\). Clearly, this fact has some bearing on the
  general picture described in \cite[Section 4]{mt13}. However, we will not
  explore this connection further in this paper.
\end{rem}
Finally let us consider the first ``new'' case, \(r=3\). The third zigzag
quiver is
\[ Z_{3} = \xymatrix{
  1 \ar@(ul,dl)[]_{a} \ar[r]|{y} & 2 \ar@/_/[l]_{x_{1}}
  \ar@/^/[l]^{x_{2}}}. \]
As \(q_{3}=2\), triangular symplectomorphisms of \(\bb{C}\Q_{3}\) are
indexed by a necklace word \(f\) in the linear space
\(\ol{\DR}^{0}(\falg{a,b_{11},b_{21}})\), where \(b_{11} = x_{1}y\) and
\(b_{21} = x_{2}y\). The matrix in \(\GL_{3}(\mathcal{A}_{1})\) corresponding
to \(\Lambda(f)\) is
\[ N^{\Lambda(f)} =
\begin{pmatrix}
  1 & 0 & 0\\
  \frac{\partial f}{\partial b_{11}} & 1 & \frac{\partial f}{\partial b_{21}}\\
  0 & 0 & 1
\end{pmatrix}. \]
When restricted to \(P_{3}\subseteq \TAut(\bb{C}\Q_{3};c_{3})\), the map \(N\)
becomes an isomorphism with \(\GL_{3}(\bb{C}[a])\). For instance one can
compute that the generic upper unitriangular matrix
\[
\begin{pmatrix}
  1 & p_{12}(a) & p_{13}(a)\\
  0 & 1 & p_{23}(a)\\
  0 & 0 & 1
\end{pmatrix}
\]
in \(\GL_{3}(\bb{C}[a])\) corresponds to the symplectomorphism of
\(\bb{C}\Q_{3}\) defined by
\[
\begin{cases}
  a^{*}\mapsto a^{*} + \frac{\partial }{\partial a}(p_{23}(a)x_{2}y) -
  \frac{\partial }{\partial a}(p_{12}(a)y^{*}x_{1}^{*}) - p_{23}(a)
  \frac{\partial }{\partial a}(p_{12}(a)x_{2}x_{1}^{*}) + \frac{\partial
  }{\partial a}(p_{13}(a)x_{2}x_{1}^{*})\\
  y\mapsto y + x_{1}^{*}p_{12}(a)\\
  x_{2}^{*}\mapsto x_{2}^{*} + yp_{23}(a) + x_{1}^{*}p_{13}(a)\\
  x_{1}\mapsto x_{1} + p_{12}(a)y^{*} + p_{12}(a)p_{23}(a)x_{2} -
  p_{13}(a)x_{2}\\
  y^{*}\mapsto y^{*} + p_{23}(a)x_{2}
\end{cases}
\]
One also has similar expressions for automorphisms corresponding to matrices
in \(\GL_{3}(\bb{C}[a^{*}])\).

\appendix
\section{Proof of theorem \ref{teo:sol-char-eq}}
\label{s:app}

Let \(G\) stand for the subset \(\{g_{0}, \dots, g_{n}\}\) of generators of
\(\mathcal{A}_{1}\) which appear in equation \eqref{eq:char-eq}. We start by
proving some general results about an \((n+1)\)-tuple \(\vec{u} = (u_{0},
\dots, u_{n})\) of elements of \(\mathcal{A}_{1}\) that satisfies equation
\eqref{eq:char-eq}.
\begin{lem}
  \label{lem:words-G}
  For each \(i=0\dots n\), each word of nonzero length in the support%
  \footnote{Recall that the \emph{support} of a non-commutative polynomial
    \(w\in \falg{G}\), denoted \(\supp w\), is simply the set of monomials (or
    words) in \(G\) which appear in \(w\).} of \(u_{i}\) begins and ends with
  a letter in \(G\).
\end{lem}
\begin{proof}
  Without loss of generality, let us take \(i=0\) and write \(u_{0} = \ell m +
  \chi\), where \(\ell\) is a generator, \(m\) is a word (possibly of length
  zero) and \(\chi\in \mathcal{A}_{1}\) does not contain \(-\ell m\). Equation
  \eqref{eq:char-eq} becomes
  \begin{equation}
    \label{eq:23}
    g_{0}\ell m + g_{0}\chi - \ell m g_{0} - \chi g_{0} + g_{1}u_{1} -
    u_{1}g_{1} + \dots + g_{n}u_{n} - u_{n}g_{n} = 0.
  \end{equation}
  Suppose \(\ell\notin G\). Then the term \(-\ell m g_{0}\) cannot be canceled
  by the first two terms (because \(\ell\neq g_{0}\)), nor by the fourth (by
  the hypothesis on \(\chi\)), nor by a term \(g_{i}u_{i}\) (because
  \(\ell\neq g_{i}\)), nor by a term \(-u_{i}g_{i}\) (because \(g_{i}\neq
  g_{0}\) for every \(i\neq 0\)). This contradicts equation \eqref{eq:23},
  hence it must be that \(\ell\in G\). Similarly, writing \(u_{0} = m\ell +
  \chi\) the same reasoning implies that \(\ell\in G\).
\end{proof}
\begin{lem}
  \label{lem:cyc-ii}
  For each \(i=0\dots n\), \(u_{i}\) contains a word of the form \(mg_{i}\) if
  and only if it contains a word of the form \(g_{i}m\).
\end{lem}
\begin{proof}
  Again let us fix \(i=0\). Notice that the lemma trivially holds if \(m =
  g_{0}^{s}\) for some \(s\in \bb{N}\), hence we can suppose that \(m\)
  contains a letter different from \(g_{0}\). Write \(u_{0} = mg_{0} + \chi\)
  for some \(\chi\in \mathcal{A}_{1}\) not containing \(-mg_{0}\). Equation
  \eqref{eq:char-eq} becomes
  \[ g_{0}mg_{0} + g_{0}\chi - m g_{0}g_{0} - \chi g_{0} + g_{1}u_{1} -
  u_{1}g_{1} + \dots + g_{n}u_{n} - u_{n}g_{n} = 0. \]
  Consider the term \(g_{0}mg_{0}\). It cannot be canceled by the second term
  (by hypothesis) and it cannot be canceled by terms coming from other
  commutators (because \(g_{i}\neq g_{0}\) for every \(i\neq 0\)). Neither can
  it be canceled by \(-mg_{0}g_{0}\), because as soon as \(m\) contains a
  generator different from \(g_{0}\) we have that \(mg_{0}-g_{0}m\neq 0\).
  Then it must be canceled by a term in \(-\chi g_{0}\), i.e. \(g_{0}m\in
  \supp \chi\). The opposite implication is shown in the same way.
\end{proof}
From these results it follows that each \(u_{i}\) is a word in \(\falg{G}\),
the free algebra generated by \(G\).
\begin{lem}
  \label{lem:cyc-ij}
  For each \(i=0\dots n\) and for each \(j=0\dots n\) with \(j\neq i\),
  \(u_{i}\) contains a word of the form \(mg_{j}\) if and only if \(u_{j}\)
  contains a word of the form \(g_{i}m\).
\end{lem}
\begin{proof}
  Fix \(i=0\), \(j=1\) and suppose \(u_{0} = mg_{1} + \chi\); equation
  \eqref{eq:char-eq} becomes
  \[ g_{0}mg_{1} + g_{0}\chi - mg_{1}g_{0} - \chi g_{0} + g_{1}u_{1} -
  u_{1}g_{1} + \sum_{k=2}^{n} (g_{k}u_{k} - u_{k}g_{k}) = 0. \]
  Then, reasoning as in the previous proofs, we see that the term
  \(g_{0}mg_{1}\) can be canceled only by \(-u_{1}g_{1}\), hence \(g_{0}m\)
  appears in \(u_{1}\). The converse is proved in the same way.
\end{proof}
Now let \(S_{\vec{u}}\) denote the (finite) set consisting of pairs \((i,w)\)
where \(i=0\dots n\) and \(w\in \supp u_{i}\). By the previous lemmas, given
an element \((i,w)\in S_{\vec{u}}\) with \(w = mg_{j}\) the pair
\((j,g_{i}m)\) is again in \(S_{\vec{u}}\). By repeating this operation a
sufficient number of times we must eventually return to the original pair
\((i,w)\), which means that on \(S_{\vec{u}}\) there is an action of a cyclic
group of some finite order. As this argument also holds for every other pair,
we conclude that the set \(S_{\vec{u}}\) is partitioned into a finite number
of orbits of cyclic groups \(C_{k_{1}}, \dots, C_{k_{s}}\) for some integers
\(k_{1}, \dots, k_{s}\in \bb{N}\).
\begin{ex}
  \label{ex:bab}
  Let \(n=1\), \(G = \{a,b\}\) and consider the pair \(\vec{u} = (bab+bb,
  aba+ab+ba)\) satisfying \([a,u_{0}] + [b,u_{1}] = 0\). Then
  \[ S_{\vec{u}} = \{(0,bab), (0,bb), (1,aba), (1,ab), (1,ba)\}. \]
  There are two orbits, the first given by
  \[ (0,bab)\rightarrow (1,aba)\rightarrow (0,bab) \]
  and the second by
  \[ (0,bb)\rightarrow (1,ab)\rightarrow (1,ba)\rightarrow (0,bb). \]
\end{ex}
\begin{proof}[Proof of theorem \ref{teo:sol-char-eq}]
  One direction is immediate: if \(u_{i} = \frac{\partial f}{\partial
    g_{i}}\) then by \cite[Prop. 11.5.4]{ginz05} we have that
  \[ \sum_{i=0}^{n} \comm{\frac{\partial f}{\partial g_{i}}}{g_{i}} = 0 \]
  which is equation \eqref{eq:char-eq}. For the converse, let \(S_{\vec{u}}\)
  be as above. Take \((k,w)\in S_{\vec{u}}\) and let \(O_{w}\) be the orbit of
  \((k,w)\) under the action of the corresponding cyclic factor. For each
  \(i=0\dots n\), define \(u_{i}^{w}\) to be the sum of monomials in \(u_{i}\)
  that belongs to \(O_{w}\). Then the necklace word \(f^{w}\deq g_{k}w\) is
  such that
  \[ \frac{\partial f^{w}}{\partial g_{i}} = c_{w} u_{i}^{w} \]
  for every \(i\), where \(c_{w}\in \bb{N}\) is some combinatorial factor
  which depends only on the structure of the particular monomial \(w\) chosen.
  By repeating this construction for every orbit and summing all the necklace
  words so obtained (divided by \(c_{w}\) if necessary) we get the required
  primitive for \(\vec{u}\).
\end{proof}
In example \ref{ex:bab} above, we can take for instance the pair \((0,bab)\)
in the first orbit, which produces the necklace word \(f_{1} = abab\) with
\[ \frac{\partial }{\partial a} f_{1} = 2bab \quad\text{ and }\quad
\frac{\partial }{\partial b} f_{1} = 2aba,  \]
and the pair \((1,ba)\) in the second orbit, which gives the necklace word
\(f_{2} = bba\) with
\[ \frac{\partial }{\partial a} f_{2} = bb \quad\text{ and }\quad
\frac{\partial }{\partial b} f_{2} = ba + ab. \]
Then it is immediate to check that \(f = \frac{1}{2}abab + bba\) is in fact a
primitive for the pair \(\vec{u} = (bab+bb, aba+ab+ba)\).

\section{Other possible choices of the non-commutative symplectic form}
\label{other-nc}

In the introduction we pointed out that for each \(r>1\) there are many
quivers whose double coincides with \(\Q_{r}\); choosing a quiver different
from \(Z_{r}\) will alter the non-commutative symplectic form on
\(\bb{C}\Q_{r}\) and consequently its group of (tame) symplectomorphisms.

For an extreme example of this phenomenon, suppose that we make the
(apparently very natural) choice of taking the unstarred arrows in \(\Q_{r}\)
all oriented in the same direction, e.g. by selecting the arrows \(d_{1},
\dots, d_{r}\) and putting \(d_{\alpha}^{*} = b_{\alpha}\) for every \(\alpha
= 1\dots r\). The moment element in \(\bb{C}\Q_{r}\) determined by this choice
is
\[ \tilde{c}_{r} = [a,a^{*}] + \sum_{\alpha=1}^{r}
[d_{\alpha},d_{\alpha}^{*}], \]
which can be decomposed as \(\tilde{c}_{r}^{1} + \tilde{c}_{r}^{2}\) with
\(\tilde{c}_{r}^{1} = [a,a^{*}] + \sum_{\alpha} d_{\alpha}d_{\alpha}^{*}\) and
\(\tilde{c}_{r}^{2} = -\sum_{\alpha} d_{\alpha}^{*}d_{\alpha}\). 

Now let \(\psi\) be an automorphism of \(\bb{C}\Q_{r}\) fixing the arrows
\((a, d_{1}, \dots, d_{r})\). Let us write as usual \(\psi(a^{*}) = a^{*} +
h\) for some \(h\in \mathcal{A}_{1}\) and \(\psi(d_{\alpha}^{*}) =
d_{\alpha}^{*} + s_{\alpha}\) for some \(s_{1}, \dots, s_{r}\in
\mathcal{A}_{21}\). Then \(\psi\) fixes \(\tilde{c}_{r}^{2}\) if and only if
\begin{equation}
  \label{eq:22}
  \sum_{\alpha=1}^{r} s_{\alpha}d_{\alpha} = 0.
\end{equation}
As \(\mathcal{A}_{21}\) is a free right \(\mathcal{A}_{1}\)-module we can
write \(s_{\alpha} = \sum_{\beta} b_{\beta}u_{\beta\alpha}\) for some unique
coefficients \(u_{\beta\alpha}\in \mathcal{A}_{1}\), so that equation
\eqref{eq:22} becomes
\[ \sum_{\alpha,\beta=1}^{r} b_{\beta} u_{\beta\alpha} d_{\alpha} = 0. \]
Multiplying by \(b_{1}\) from the right and by \(d_{1}\) from the left we get
\[ \sum_{\alpha,\beta=1}^{r} e_{1\beta} u_{\beta\alpha} e_{\alpha 1} = 0. \]
As the \(e_{\alpha\beta}\) are independent in \(\mathcal{A}_{1}\), this
implies that \(\sum_{\beta} e_{1\beta} u_{\beta\alpha} = 0\) for every
\(\alpha = 1\dots r\), which in turn implies \(u_{\beta\alpha}=0\) for every
\(\alpha,\beta = 1\dots r\). We conclude that with this choice of the
non-commutative symplectic form every triangular symplectomorphism fixes all
the arrows \(d_{1}^{*}, \dots, d_{r}^{*}\). Clearly, this makes the resulting
group of tame symplectomorphisms hardly useful.

In general, reasoning as in Section \ref{s:nc} one can show that for any given
choice of a quiver \(Q_{r}\) such that \(\ol{Q}_{r} = \Q_{r}\) the
corresponding triangular symplectomorphisms of \(\bb{C}\Q_{r}\) will be
labeled by a (non-scalar) necklace word in the free algebra generated by all
the possible cycles \(1\to 1\) of length at most \(2\) that one can make in
\(Q_{r}\) (i.e. using only the unstarred arrows in \(\Q_{r}\)). In this sense
choosing the zigzag quiver \(Z_{r}\) makes the group
\(\TAut(\bb{C}\Q_{r};c_{r})\) the largest possible.

Another interesting choice, which gives the minimal (non-trivial) result, is
to take a single arrow going in one direction, e.g. \(x\deq d_{1}\), and all
the other arrows going in the opposite one, say \(y_{k}\deq b_{k+1}\) for
every \(k=1\dots r-1\). Let us call \(Q'_{r}\) the quiver obtained in this way
(notice that \(Q'_{2}\) is again the Bielawski-Pidstrygach quiver). Then we
have the \(r-1\) cycles \(\ell_{k}\deq xy_{k}\) in \(Q_{r}'\), so that
triangular symplectomorphisms are indexed by elements of the linear space
\[ \ol{\DR}^{0}(\falg{a, \ell_{1}, \dots, \ell_{r-1}}), \]
the automorphism corresponding to a necklace word \(f\) being given by
\[ (a^{*},x^{*},y_{k}^{*})\mapsto (a^{*} + \frac{\partial f}{\partial a},
x^{*} + \sum_{k=1}^{r-1} y_{k} \frac{\partial f}{\partial \ell_{k}}, y^{*}_{k}
+ \frac{\partial f}{\partial \ell_{k}}x). \]
It is not difficult to see that the resulting group of tame symplectomorphisms
of \(\bb{C}\ol{Q}'_{r} = \bb{C}\Q_{r}\) contains the group \(\mathcal{P}_{r}\)
defined in Section \ref{s:defP}. In fact, it seems that most (if not all) of
the results of the present paper do not actually depend on the particular
choice of the family of quivers \((Z_{r})_{r\geq 1}\), as long as one rules
out the trivial case shown at the beginning of this appendix.

\section*{Acknowledgements}

This work has been supported by FAPESP (Fundação de Amparo à Pesquisa do
Estado de São Paulo), thanks to the post-doctoral grant 2011/09782-6.


\providecommand{\bysame}{\leavevmode\hbox to3em{\hrulefill}\thinspace}
\providecommand{\MR}{\relax\ifhmode\unskip\space\fi MR }
\providecommand{\MRhref}[2]{%
  \href{http://www.ams.org/mathscinet-getitem?mr=#1}{#2}
}
\providecommand{\href}[2]{#2}

\end{document}